\documentclass[10pt,a4paper]{amsart}
\usepackage{amsmath,mathtools,amssymb,amsopn,amsthm,graphicx,diagrams,MnSymbol,centernot,stmaryrd}

\title{Taxotopy Theory of Posets I: van Kampen Theorems}

\author{Amit Kuber and David Wilding}

\begin{document}

\newtheorem{definition}{Definition}[section]
\newtheorem{definitions}[definition]{Definitions}
\newtheorem{lemma}[definition]{Lemma}
\newtheorem{pro}[definition]{Proposition}
\newtheorem{theorem}[definition]{Theorem}
\newtheorem{cor}[definition]{Corollary}
\newtheorem{cors}[definition]{Corollaries}
\theoremstyle{remark}
\newtheorem{notation}[definition]{Notation}
\theoremstyle{remark}
\newtheorem{example}[definition]{Example}
\theoremstyle{remark}
\newtheorem{examples}[definition]{Examples}
\theoremstyle{remark}
\newtheorem{rmk}[definition]{Remark}
\theoremstyle{definition}
\newtheorem{que}[definition]{Question}
\theoremstyle{definition}
\newtheorem{por}[definition]{Porism}

\renewenvironment{proof}{\noindent {\bf{Proof.}}}{\hspace*{3mm}{$\Box$}{\vspace{9pt}}}

\newcommand\Set{\operatorname{Set}}
\newcommand\Cat{\operatorname{Cat}}
\newcommand\Gal{\operatorname{Gal}}
\newcommand\Pos{\operatorname{Pos}}
\newcommand\Pre{\operatorname{Pre}}
\newcommand\Top{\operatorname{Top}}
\newcommand\Alx{\operatorname{Alx}}
\newcommand\SC{\operatorname{SimpComp}}
\newcommand\Ob{\operatorname{Ob }}
\newcommand\id[1]{\mathrm{id}_{#1}}
\newcommand\im[1]{\operatorname{im}(#1)}
\newcommand\Adj[1]{\operatorname{Adj}(#1)}
\newcommand\Loc[1]{\operatorname{Loc}_{\operatorname{adj}}(#1)}
\newcommand\Subw[1]{\operatorname{Sub}_w(#1)}
\newcommand\Subs[1]{\operatorname{Sub}_s(#1)}
\newcommand\Cont[1]{\operatorname{Cont}_{#1}}
\newcommand\vsim[1]{\,\mid\!\sim_{#1}\,}
\newcommand\simv[1]{\,\sim\!\mid_{#1}\,}


\begin{abstract}
Given functors $F,G:\mathcal C\to\mathcal D$ between small categories, when is it possible to say that $F$ can be ``continuously deformed'' into $G$ in a manner that is not necessarily reversible? In an attempt to answer this question in purely category-theoretic language, we use adjunctions to define a `taxotopy' preorder $\preceq$ on the set of functors $\mathcal C\to\mathcal D$, and combine this data into a `fundamental poset' $(\Lambda(\mathcal C,\mathcal D),\preceq)$.

The main objects of study in this paper are the fundamental posets $\Lambda(\mathbf 1,P)$ and $\Lambda(\mathbb Z,P)$ for a poset $P$, where $\mathbf 1$ is the singleton poset and $\mathbb Z$ is the ordered set of integers; they encode the data about taxotopy of points and chains of $P$ respectively. Borrowing intuition from homotopy theory, we show that a suitable cone construction produces `null-taxotopic' posets and prove two forms of van Kampen theorem for computing fundamental posets via covers of posets.
\end{abstract}

\maketitle

\section{Introduction and Motivation}
Let $\Cat$, $\Pos$ and $\Pre$ denote the categories of small categories and functors, posets and monotone maps, and preorders and monotone maps respectively. Also let $\Top$, $\SC$ denote the category of topological spaces and continuous functions, simplicial complexes and simplicial maps respectively.

One can think of a small category as a topological space via the classifying space functor $B:\Cat\to\Top$ (see \cite[\S IV.3]{Weibel}), and thus adjectives like connected, homotopy equivalent etc., attributed to topological spaces can also be used for categories. For the full subcategory $\Pos$ of $\Cat$, this classifying space functor factors through $\SC$, for given $P\in\Pos$, the space $BP$ is the geometric realization of the order complex, $\Delta(P)$, which is a simplicial complex. See \cite{McCord} for the details of the construction of $\Delta(P)$ for finite posets.

Adjoint functors are essential features of categories that distinguish them from groupoids. Unfortunately, adjoint functors become homotopy equivalences under the functor $B$, and hence the distinction between left and right adjoints disappears. This paper is motivated by the following question.
\begin{que}
Is there a purely category-theoretic way of doing topology and/or homotopy theory with categories and functors, especially with posets and monotone maps, which distinguishes between left and right adjoints?
\end{que}

While studying a model structure on the category of monotone Galois connections between posets, we pondered this question with the aim of classifying posets by associating a fundamental object with each of them. A preliminary examination of the question suggested that, if successful, we would obtain a not necessarily reversible version of homotopy between two functors. We indeed succeeded in associating with each poset a fundamental poset, but the domain of this assignment extended to include all small categories.

The question is very general and several authors have developed directed versions of homotopy theory; the motivation ranges from curiosity to applications, especially in computer science, as is the case with any theory. Goubault \cite{Goubault} and coauthors laid the foundations for directed homotopy theory by developing geometric models for concurrency---a branch of computer science. Several approaches to directed topology were made which include preordered topological spaces, locally preordered topological spaces and Kelly's bitopological spaces. These ideas were unified in the notion of a directed topological space. Grandis developed directed homotopy theory of directed topological spaces in \cite{GrandisI, GrandisII}, where one associates a fundamental monoid or a fundamental category with such a space. See \cite{GrandisBook} for an overview of the directed algebraic topology and \cite{FajRos} for a categorical approach to directed homotopy theory.

The fundamental poset we define shares many properties with the fundamental group in the homotopy theory of topological spaces. For example, we will define what a null-taxotopic poset is (Definition \ref{nulltaxotopy}), compute fundamental posets for total orders (Corollary \ref{lambdatotalorder}) and for cones over posets (Theorem \ref{LCone}), show that a cone construction always produces null-taxotopic posets (Proposition \ref{listnulltax}), and also prove suitable van Kampen theorems for the fundamental posets (Theorems \ref{vanKampenlambda},\ref{vanKampengen}). 

This theory, which we call taxotopy theory, is still in its infancy, and we are continuing to develop it. One of the achievements of this theory is that it is possible to distinguish between two homotopy equivalent connected posets using taxotopy.

In the rest of this section we introduce the idea behind the partial order in the fundamental poset from four different contexts, ranging from semigroup theory to topology and category theory.

\subsection{Green's relations on semigroups}
Let $(M,\cdotp,1)$ be a monoid and $X$ be a left $M$-act, i.e., $X$ is a set with a map $\cdotp:M\times X\to X$ that satisfies $1\cdotp x=x$ and $a\cdotp(b\cdotp x)=(ab)\cdotp x$ for all $a,b\in M$, $x\in X$. One can define a category ${}_M\mathbf X$ with $X$ as the set of objects and $\{a\in M: a\cdotp x=y\}$ as the set of morphisms from $x$ to $y$. This category contains all the data of the $M$-action.

Recall that the orbit $\mathrm{Orb}(x)$ of $x\in X$ is the collection of all $y\in X$ such that there is some $a\in M$ with $a\cdotp x=y$. In other words, the hom-set ${}_M\mathbf X(x,y)$ is nonempty. We could define a binary relation on $X$ by $x\preceq_l y$ if and only if $y\in\mathrm{Orb}(x)$. The subscript $l$ denotes that $X$ is a left $M$-act. This relation is a preorder since it corresponds to the preorder reflection of the category ${}_M\mathbf X$. Note that $\preceq_l$ is an equivalence relation whenever $M$ is a group.

In semigroup theory, Green's relations are five equivalence relations that characterize elements of a semigroup in terms of the principal ideals they generate. See \cite{Howie} for more details.

Recall that for elements $a,b\in M$, Green's relation $L$ is defined by $a\,L\,b$ if and only if $Ma=Mb$. Choose $X=M$ in the above discussion, i.e., consider $M$ as a left regular act. Then $a\,L\,b$ if and only if both $a\preceq_l b$ and $b\preceq_l a$. Green's relation $R$ can be defined in a similar way using the preorder $\preceq_r$ obtained on right $M$-acts. Two other Green's relations, $H$ and $D$, can be defined in terms of $L$ and $R$, and hence in terms of $\preceq_l$ and $\preceq_r$. Finally, the relation $J$ is defined by $a\,J\,b$ if and only if $MaM=MbM$. Note that $c\in MaM$ if and only if there is some $c'$ such that $a\preceq_l c'\preceq_r c$ if and only if there is some $c''$ such that $a\preceq_r c''\preceq_l c$. Therefore $J$ can also be defined in terms of $\preceq_l$ and $\preceq_r$. In fact, the preorders carry more information than the equivalence relations!

For any small category $\mathcal C$, the monoid $\mathrm{End}(\mathcal C)$ of endofunctors of $\mathcal C$ acts on $\Ob{\mathcal C}$, say on the left. Construct the category ${}_{\mathrm{End}(\mathcal C)}\mathcal C$ as described above and obtain a preorder $\preceq_l$ on $\Ob{\mathcal C}$. Note that in a slice category ${}_{\mathrm{End}(\mathcal C)}\mathcal C/C$ there could be two objects with the same label but the labels corresponding to automorphisms of the category are unique. We want to avoid working within the strict setting of automorphisms, but we would still like to be able trace the source of a morphism whose target and label is known.

\subsection{Local categories of adjunctions}
Suppose $\mathcal C$ is a small category. Let $\Adj{\mathcal C}$ denote the monoid of all pairs of adjoint functors $F=(F^*\dashv F_*)$ on $\mathcal C$.
\begin{definition}
The \textbf{local category of adjunctions} on $\mathcal C$, denoted $\Loc{\mathcal C}$, has as objects the objects of $\mathcal C$ and as morphisms $C\to D$ those adjunctions $F\in\Adj{\mathcal C}$ which satisfy $F^*D=C$ and $F_*C=D$. We will use the notation $F\models C\preceq D$ to denote $F\in\Loc{\mathcal C}(C,D)$.
\end{definition}
We follow the convention of topos theory and additive category theory, where geometric morphisms are given the direction of the right adjoint, and thus choose the direction of the right adjoint for the new map. Note that our morphisms are indeed pairs of adjoint functors and not just their isomorphism classes. We need to assume that $\mathcal C$ is small to guarantee that each hom-set in $\Loc{\mathcal C}$ is indeed a set.

We can think of the monoid $\Adj{\mathcal C}$ as a single object category and denote the category by the same notation. There is an obvious faithful functor $\Loc{\mathcal C}\to\Adj{\mathcal C}$. From this viewpoint we may think of the local category of adjunctions as a ``resolution'' of the monoid of adjunctions.

As discussed at the end of the previous section, one could in fact construct the local category of endofunctors in a similar fashion to obtain a resolution of the monoid of all endofunctors of the category. Since adjunctions allow us to ``go back'' along the arrow, we cannot have two arrows with the same label and target but with distinct sources. Hence we restrict our attention only to the local category of adjunctions.

Locally in a hom-set $\Loc{\mathcal C}(C,D)$, the labels on parallel arrows do not really matter since their role is identical in the sense that each such $F$ satisfies $F\models C\preceq D$. This motivates us to take the preorder reflection of the category $\Loc{\mathcal C}$, which we denote by $\lambda(\mathcal C)$.

\subsection{Adjoint topology on a category}
Given $\mathcal C\in\Cat$, we define a collection of subsets of $\Ob\mathcal C$ that are closed under the operation of ``going forward'' along an adjunction.
\begin{equation*}
\tau:=\{U\subseteq\Ob\mathcal C\mid \forall C\in U \forall F:\mathcal C\to\mathcal C ((\exists F^*\dashv F\wedge C\in F^*(\mathcal C))\implies F(C)\in U)\}.
\end{equation*}
Each $U\in\tau$ is closed under isomorphisms. It can be easily checked that this set is closed under arbitrary unions and finite intersections. Hence $\tau$ defines a topology on $\Ob\mathcal C$; call it the \textbf{adjoint topology}. Moreover, the collection $\tau$ is also closed under arbitrary intersections. This means that $\tau$ is an Alexandroff topology on $\Ob\mathcal C$. Alexandroff topologies were introduced in \cite{Alex} by Alexandroff.

Steiner showed in \cite[Theorem~2.6]{Steiner} that an Alexandroff space can be thought of as a preorder which leads to a well-known equivalence of categories between the category $\Pre$ and the full subcategory of $\Top$ consisting of Alexandroff spaces, denoted $\Alx$ (also known as $A$-spaces). As a consequence we obtain a preorder, $\preceq$, known as the specialization preorder, on $\Ob\mathcal C$. The poset quotient $(\Ob\mathcal C,\preceq)/(\preceq\cap\succeq)$ corresponds to the $T_0$-quotient of the topological space $(\Ob\mathcal C,\tau)$.

We have not yet made clear how the above-defined topology distinguishes between left and right adjoints. Also, we would like to understand the preorder $\preceq$ in a better way. If we unravel the definition of an open set in $\tau$ and the categorical equivalence mentioned in the above paragraph, then we obtain the following definition of the preorder.
\begin{equation*}
C\preceq D \iff \exists F\in\Adj{\mathcal C}\mbox{ such that }F^*D=C\mbox{ and }F_*C=D.
\end{equation*}
This shows that it is not a coincidence that we use the same notation for two preorders on $\Ob{\mathcal C}$---one obtained from the local category of adjunctions and the other obtained from the adjoint topology---but they are in fact identical!

\subsection{Category of adjunctions}\label{catofadj}
Let $\mathrm{ADJ}$ denote the large category with pairs of adjoint functors between small categories, $F^*:\mathcal C'\rightleftarrows\mathcal C:F_*$, as objects and pairs of functors $(H:\mathcal D'\to\mathcal C',K:\mathcal D\to\mathcal C)$ as morphisms from $G^*:\mathcal D'\rightleftarrows\mathcal D:G_*$ to $F^*:\mathcal C'\rightleftarrows\mathcal C:F_*$ whenever $K\circ G^*\simeq F^*\circ H$, $H\circ G_*\simeq F_*\circ K$ and $H\circ\eta'=\eta\circ H$, where $\eta',\eta$ are the units of adjunctions. Equivalent characterizations of this category are discussed in \cite[\S IV.7]{MacLane}.

Choose $\mathcal D,\mathcal D'$ to be the terminal category $\mathbf 1$, $G^*\dashv G_*$ to be the identity adjunction, and $\mathcal C'=\mathcal C$. Further let $H,K$ pick out the objects $C',C\in\mathcal C$ respectively. Then $(C',C):\id{\mathbf 1}\to F$ is a morphism in $\mathrm{ADJ}$ if and only if $F\models C\preceq C'$.

This is perhaps the simplest way to define the preorder $\preceq$ on $\Ob\mathcal C$ but other ways give more intuition of why this definition is interesting from various perspectives.

The case when the domain of a morphism is not the identity adjunction is also important and will form the basis for the ``taxotopy theory'' we develop in this paper. After defining the taxotopy relation between functors in the next section, we will explain, from the homotopy theory point of view, why this binary relation can be thought of as a way to ``deform'' a functor into another in a ``continuous'' manner.

\subsection{Notation}
We need to set up some notation to work with posets throughout the rest of this paper.

The partial order relation in a poset $P$ will be written using the symbol $\leq$; the exceptions being fundamental posets, where the partial order is the taxotopy relation denoted by $\preceq$. We do not distinguish between the element of a preorder and its equivalence class in the posetal reflection.

Duality in the context of posets always refers to the functor $(-)^{op}:\Pos\to\Pos$ that takes a poset $P$ to the opposite poset $P^{op}$.

A subset $Q$ of a poset $P$ is a cutset if each maximal chain in $P$ intersects $Q$.

For elements $p,q\in P$, the upper set of $p$ is $p^\uparrow:=\{p'\in P:p\leq p'\}$ and the lower set of $q$ is $q^\downarrow:=\{q'\in P:q'\leq q\}$. If $p\leq q$ then the closed interval $[p,q]$ is $p^\uparrow\cap q^\downarrow$ and the open interval $(p,q)$ is $[p,q]\setminus\{p,q\}$.

The totally ordered sets of natural numbers and integers are denoted by $\mathbb N$ and $\mathbb Z$ respectively. We assume that $0\in\mathbb N$.

For positive $n\in\mathbb N$, the ordered set $0<1<\cdots<n-1$ is denoted by $\mathbf n$.

The notation $\sqcup$ denotes disjoint union. For positive $n\in\mathbb N$, the notation $n\cdotp P$ denotes the disjoint union of $n$ copies of $P$.

We will sometimes write the composition of two arrows in a category using juxtaposition; other times we will use $\circ$.

See \cite{Hatcher} and \cite{MacLane} for background in algebraic topology and category theory respectively.

\subsection*{Acknowledgments:} Both authors wish to express sincere thanks to Harold Simmons for his positive approach towards this project and for sharing innumerable discussions with us. Thanks also go to Mark Kambites for suggesting the study of $L(-)$ since $\mathbb Z$-chains in a poset are analogues of paths in a topological space. The first author also wishes to express thanks to Mike Prest for his generous support for research visits to Manchester.

The first author was funded by the Italian project FIRB 2010 and a doctoral scholarship by School of Mathematics, University of Manchester. The second author was funded by a University of Manchester Faculty of Engineering and Physical Sciences Dean's Award while studying at the University of Manchester.

\section{Taxotopy of monotone maps}
We only work with posets and monotone maps throughout the rest of this paper; in most cases the generalization to categories and functors is obvious via the category $\mathrm{ADJ}$ of Section \ref{catofadj}.

Let $P,Q$ be posets and let $h,k:P\to Q$ be any two monotone maps. Using intuition borrowed from homotopy theory of topological spaces, we would like to express when the map $k$ can be ``continuously'' deformed into the map $h$. As a consequence of whatever definition of deformation we give, we are clearly expecting the continuous function $Bk$ to be homotopic to $Bh$. The latter is an equivalence relation but the former need not be! In fact, we will define a preorder on the set $\Pos(P,Q)$ to describe the relation ``can be deformed into''.

All the discussion and motivation in the previous section was aimed at the following definition.
\begin{definition}
For $P,Q\in\Pos$ and $h,k\in\Pos(P,Q)$, we say that $k$ is \textbf{taxotopic} to $h$, written $k\preceq h$, if there are $f\in\Adj P$ and $g\in\Adj Q$ such that $k\circ f^*=g^*\circ h$ and $h\circ f_*=g_*\circ k$. We express this as $(f,g)\models k\preceq h$, read $f,g$ witness $k\preceq h$, if we want to emphasize the adjunctions.
\end{definition}
The binary relation $\preceq$ defined above is indeed a preorder for reflexivity follows from the fact that identity map is self-adjoint, and transitivity follows from the fact that composition of left (resp. right) adjoints is again a left (resp. right) adjoint.
\begin{rmk}
The term \textit{taxotopy} is derived from two Greek words, \textit{taxis} ($\tau\alpha\xi\iota\sigma$ in Greek) meaning \textit{order}, and \textit{topos} ($\tau o\pi o\sigma$) meaning \textit{space}, just as the word \textit{homotopy} is derived from the Greek words \textit{homos}, meaning \textit{similar}, and \textit{topos}. The credit for this name goes to Nathana\"{e}l Mariaule.
\end{rmk}

Observe that in the two commutative diagrams, either both left adjoints appear or both right adjoints. The direction of the taxotopy order is the same as the direction of the right adjoints. This sense of direction/order disappears when one passes to the classifying spaces $BP$ and $BQ$. There one observes that $Bh,Bk:BP\to BQ$ are homotopic as continuous maps. In this sense, taxotopy provides a way of classifying homotopy equivalent (monotone) maps. It should be noted that since we did not fix a point in any poset, we can only refer to the homotopy between paths. Moreover the posets, and hence their classifying spaces, could have more than one connected component.

Note the existential clause in the definition above, i.e., the taxotopy preorder depends on the existence of adjunctions on the two posets. Thus the relation $\preceq$ is not symmetric in general. Given $h,k\in\Pos(P,Q)$, if $h\preceq k$ and $k\preceq h$, then we say that $h$ and $k$ are \textbf{taxotopy equivalent} and express this as $h\approx k$.

\begin{rmk}
Recall that a homotopy between two paths $h,k:[0,1]\to X$ in a topological space $X$ is a continuous function $H:[0,1]^2\to X$ such that $H(t,0)=h(t)$ and $H(t,1)=k(t)$ for each $t\in[0,1]$. So there is square shaped picture in $X$ that represents continuous deformation of the bottom line $h$ into the top line $k$. A taxotopy between two monotone maps $h,k:P\to Q$ is also represented by a (bi-commutative) square representing deformation of the bottom map $h$ into the top map $k$. Thus in a very loose sense the taxotopy relation could be seen as a discrete ordered version of homotopy equivalence relation.
\end{rmk}

An adjunction on a poset is also known as a monotone Galois connection. (We shall never talk about antitone Galois connections and hence we drop the adjective monotone.) Given a Galois connection $f\in\Adj P$, by definition, we have $f_*f^*(p)\geq p$ and $f^*f_*(p)\leq p$ for each $p\in P$. The functions $f_*f^*:P\to P$ and $f^*f_*:P\to P$ are closure and interior operators respectively. The closure operator has $f_*(P)$ as its fixed set while the interior operator has $f^*(P)$ as its fixed set. Moreover, the adjunction induces a bijection $f^*:f_*(P)\rightleftarrows f^*(P):f_*$ between these fixed sets.

Let us try to understand the definition of taxotopy in a bit more detail. For $h,k:P\to Q$, the relation $k\preceq h$ describes that for some $f\in\Adj P$, $g\in\Adj Q$, the map $h$ restricts to a map $f_*P\to g_*Q$ between fixed sets of closure operators, and the map $k$ restricts to a map $f^*P\to g^*Q$ between fixed sets of interior operators. Moreover, since the fixed sets of closure and interior operators are isomorphic, the restrictions of $h$ and $k$ are the same functions up to these relabeling (i.e., isomorphisms). Since closure operators are extensive (i.e., they make elements bigger) while interior operators are intensive (i.e., they make elements smaller), the direction of the taxotopy order $k\preceq h$ is intuitive.

\section{Fundamental posets}
For posets $P,Q$, the hom-set $\Pos(P,Q)$ naturally acquires a partial order structure under pointwise order. We will mostly forget about this poset structure and use the notation to denote the underlying set.
\begin{definition}
The \textbf{fundamental poset} of the pair $\langle P,Q\rangle$, denoted $\Lambda(P,Q)$, is defined to be the posetal reflection of the taxotopy preorder $(\Pos(P,Q),\preceq)$.
\end{definition}
We will use the notation $\preceq$ to denote both the preorder and the partial order.

We are interested in three special instances of the fundamental poset; the notations are given below.
\begin{notation}
For $P\in\Pos$, $\lambda(P):=\Lambda(\mathbf 1,P),\ L(P):=\Lambda(\mathbb Z,P)$ and $\Lambda(P):=\Lambda(P,P)$.

The symbol $\mathord{\Diamond}$ denotes the poset consisting of four elements $\bot,a,b,\top$ related by $\bot<a,b<\top$.
\end{notation}

\begin{rmk}
For any $P\in\Pos$, the set $\Pos(\mathbf 1,P)$ is in bijection with $P$. Hence $\lambda(P)$ is precisely the poset given by the specialization order on the $T_0A$-space associated with the adjoint topology on $P$.
\end{rmk}

\begin{example}
We try to understand the poset $\lambda(\mathord{\Diamond})$. There is an evident automorphism of $\mathord{\Diamond}$ that swaps $a$ and $b$ and therefore $a\approx b$. Since right (resp. left) adjoint preserves limits (resp. colimits), it preserves the top element, $\top$ (resp. $\bot$). Hence $g^*(\bot)=\bot$ and $g_*(\top)=\top$ for all $g\in\Adj{\mathord{\Diamond}}$. This proves that $d\preceq\bot$ if and only if $d=\bot$ and $d'\geq\top$ if and only if $d'=\top$ for $d,d'\in\mathord{\Diamond}$. Existence of the adjunction $\bot:\mathord{\Diamond}\rightleftarrows\mathord{\Diamond}:\top$, where $\top$ and $\bot$ denote (by an abuse of notation) the constant maps taking those values, shows that $\bot\prec\top$.

Finally define $f\in\Adj{\mathord{\Diamond}}$ as $\langle\bot,a,b,\top\rangle\xrightarrow{f^*}\langle\bot,\bot,b,b\rangle$ and $\langle\bot,a,b,\top\rangle\xrightarrow{f_*}\langle a,a,\top,\top\rangle$. This adjunction furnishes the taxotopy relations $b\preceq\top$ and $\bot\preceq a$. We have already seen that these relations are strict. Hence $\lambda(\mathord{\Diamond})=\mathbf 3$ given by $\bot\prec a\approx b\prec\top$.
\end{example}

We will show in Section \ref{cone} that this is a special case of a much general result about posets with top and bottom elements.

It is interesting to see how Galois connections interact with the taxotopy preorder. Suppose $f\in\Adj P$. Then we can't say much about the maps $f^*$ and $f_*$ themselves but the closure and interior operators on $P$ induced by this adjunction are taxotopy related in a quite intuitive manner. In fact, if $\id P$ denotes the identity function on $P$, then $f^*f_*\preceq\id P\preceq f_*f^*$ in $\Lambda(P)$. In case $P$ has top $\top$ and bottom $\bot$ one always has an adjunction $\bot\dashv\top$, and hence $\bot\preceq\id P\preceq\top$.

\begin{rmk}\label{lambdaembeds}
Let $P,Q$ be any two posets. An element $p\in P$ can be identified with the constant map $p:Q\to P$ taking value $p$. This produces an embedding $P\to\Pos(Q,P)$. This embedding is well-behaved with respect to the taxotopy order. If $p_1\preceq p_2$ in $\lambda(P)$, then the same holds true of the constant maps with those values in $\Lambda(Q,P)$. For this last statement, we use the identity Galois connection on $Q$ and the same adjunction on $P$ that witnesses the relation in $\lambda(P)$. Moreover this relation is also reflected for one can always choose the identity adjunction on $Q$ to witness a taxotopy relation between constant maps. Hence there is an order embedding $\lambda(P)\rightarrowtail\Lambda(Q,P)$ for each $P,Q$. In particular there are order embeddings $\lambda(P)\rightarrowtail\Lambda(P)$ and $\lambda(P)\rightarrowtail L(P)$. From the discussion in the paragraph above, if $P$ has $\top$ and $\bot$, then one gets $\bot\preceq\top$ in $\lambda(P)$ by reflecting this relation in $\Lambda(P)$.
\end{rmk}

\section{Taxotopy of chains}
The totally ordered set $\mathbb Z$ of integers has many interesting properties. The image of a map $\mathbb Z\to P$ is a chain in $P$. A chain looks like the correct (discrete) order-theoretic analogue of a path in a topological space. A path is the fundamental object of study in homotopy theory of topological spaces which motivated us to study $L(P)$. We will show that some of the homotopy-theoretic properties of paths are similar to the taxotopy-theoretic properties of chains in posets.

First we characterize adjunctions on $\mathbb Z$.
\begin{lemma}\label{adjZ}
The following are equivalent for a monotone map $h:\mathbb Z\to\mathbb Z$.
\begin{enumerate}
    \item The map $h$ can be factorized as $h=h_ih_s$, where $h_i:\mathbb Z\rightarrowtail\mathbb Z$ is an injection and $h_s:\mathbb Z\twoheadrightarrow\mathbb Z$ is a surjection.
    \item There is an isomorphism between $\im h$ and $\mathbb Z$.
    \item For each $n\in\im h$, we have $|h^{-1}(n)|<\infty$.
    \item There is a left adjoint $h^*$ for $h$.
    \item There is a right adjoint $h_*$ for $h$.
\end{enumerate}
\end{lemma}
\begin{proof}
The equivalences $(1)\Leftrightarrow(2)\Leftrightarrow(3)$ are immediate.\\
To see that $(1)\Rightarrow(4)$ (resp. $(1)\Rightarrow(5)$), we construct left (resp. right) adjoint for each injective and surjective monotone map in $\Pos(\mathbb Z,\mathbb Z)$. If $h_s:\mathbb Z\twoheadrightarrow\mathbb Z$ is a surjection, then for each $n\in\mathbb Z$, we define $h_s^*(n):=\min h_s^{-1}(n)$ (resp. $h_{s*}(n):=\max h_s^{-1}(n)$). Then $h_s^*$ (resp. $h_{s*}$) is an injective monotone map which indeed is a left (resp. right) adjoint for $h_s$. On the other hand, for an injective monotone map $h_i:
\mathbb Z\rightarrowtail\mathbb Z$ we define the surjective map $h_i^*(n):=\min h_i^{-1}[n,\infty)$ (resp. $h_{i*}(n):=\max h_i^{-1}(-\infty,n]$) which is the left (resp. right) adjoint for $h_i$.\\
Finally we show that $(4)\Rightarrow(2)$. The proof of $(5)\Rightarrow(2)$ is analogous (but dual). If $h^*\dashv h$, then $hh^*\geq\id{\mathbb Z}$ and $h^*h\leq\id{\mathbb Z}$. This implies that the image of the closure operator is unbounded above while the image of the interior operator is bounded below. Since these two images are in bijection with each other, we conclude that $\im h=\im{hh^*}$ is unbounded in both directions.
\end{proof}

\begin{theorem}\label{LZ}
$\lambda(\mathbb Z)=\mathbf 1$ and $\Lambda(\mathbb Z)=L(\mathbb Z)=5\cdotp\mathbf 1$.
\end{theorem}
\begin{proof}
For each $k\in\mathbb Z$, the function $\mathbb Z\xrightarrow{(-)+k}\mathbb Z$ is an automorphism of $\mathbb Z$. Hence $n\preceq n+k$ for each $n,k\in\mathbb Z$. This shows that $\lambda(\mathbb Z)=\mathbf 1$.

For the second statement, observe that each map in $\Pos(\mathbb Z,\mathbb Z)$ falls into precisely one of the following five classes depending on the nature of its image: singleton, bounded but not singleton, bounded below but unbounded above, bounded above but unbounded below, unbounded in both directions. The inverse image of each element in the image is an interval and hence there are precisely five different types of partitions of $\mathbb Z$ into disjoint nonempty intervals. The ordered set of intervals in any such partitions is order isomorphic to precisely one of the following: singleton, finite total order of size bigger than $1$, $\mathbb N$, $\mathbb N^{op}$,$\mathbb Z$. We will show that all maps in each of these classes are taxotopy equivalent, and that there are no taxotopy relations between maps belonging to different classes.

Given $h,k:\mathbb Z\to\mathbb Z$, the relation $k\preceq h$ implies that for some adjunctions $f,g$ on $\mathbb Z$, the restrictions $h:\im{f_*}\to\im{g_*}$ and $k:\im{f^*}\to\im{g^*}$ are isomorphic. Above lemma showed that the image of any closure or interior operator on $\mathbb Z$ induced by a Galois connection is isomorphic to $\mathbb Z$. Hence we conclude that both $h$ and $k$ have images in the following four classes: finite, unbounded below but bounded above, unbounded above but bounded below, unbounded in both directions.

Now we classify the maps with finite image into two classes. Suppose both $h$ and $k$ have finite images. If $|\im h|>1$ and $|\im k|=1$, then for any $f,g\in\Pos(\mathbb Z,\mathbb Z)$ admitting adjoints, we have $|\im{gh}|>1$ and $|\im{kf}|=1$. Hence $kf\neq hg$. This shows that there are no taxotopy relations between a constant map and a non-constant map with finite image. This completes the proof that there are at least five connected components of the poset $\Lambda(\mathbb Z)$.

Now we need to show that two maps in the same class as discussed above are taxotopy equivalent. There are five different cases and we have dealt with the case of constant maps when computing $\lambda(\mathbb Z)$. In each of the remaining cases we indicate the adjunctions $f$ and $g$ on the domain and the codomain respectively, but leave the verification of the taxotopy relation $h\preceq k$ to the reader.

Suppose both $h,k$ have finite images. Let $n_1=\max h^{-1}(\min\im h)$, $n_2=\min h^{-1}(\max\im h)$, $m_1=\max k^{-1}(\min\im k)$ and $m_2=\min k^{-1}(\max\im k)$. Then $n_1<n_2$ and $m_1<m_2$. To define the Galois connection $f$, set up order isomorphisms $(-\infty,n_1]\cong(-\infty,m_1]$ and $[n_2,\infty)\cong[m_2,\infty)$. Send the interval $(n_1,n_2)$ to $m_1$ by $f^*$ while send the interval $(m_1,m_2)$ to $n_2$ by $f_*$. Define $g$ similarly by replacing $n_1,n_2,m_1,m_2$ by $\min\im h, \max\im h, \min\im k, \max\im k$ respectively.

Suppose both $h,k$ have images isomorphic to $\mathbb Z$. Choose order isomorphisms $\im h\xrightarrow{\phi}\mathbb Z\xrightarrow{\psi}\im k$. The order isomorphism $\psi\circ\phi$ can be extended to a Galois connection on $\mathbb Z$ as follows. For two consecutive elements $n_1<n_2$ in $\im h$, send the interval $(n_1,n_2)$ to $\psi\circ\phi(n_1)$ under $g^*$. Similarly for two consecutive elements $m_1<m_2$ in $\im k$, send the interval $(m_1,m_2)$ to $(\psi\circ\phi)^{-1}(m_2)$ under $g_*$. Define $f^*(n):=\min k^{-1}(\psi\circ\phi(h(n)))$ and $f_*(m):=\max h^{-1}((\psi\circ\phi)^{-1}(k(m)))$.

In the remaining two cases, we use an appropriate combination of the above two constructions of Galois connections. This completes the proof that $\Lambda(\mathbb Z)=L(\mathbb Z)=5\cdotp\mathbf 1$.
\end{proof}

\begin{rmk}
If $h,k\in\Pos(\mathbb Z,\mathbb Z)$ have images isomorphic to $\mathbb Z$, then there is another easy way of showing $k\preceq h$. Lemma \ref{adjZ} allows us to write $h=h_ih_s$ and $k=k_ik_s$, where $h_i,k_i$ are injective maps while $h_s,k_s$ are surjective maps in $\Pos(\mathbb Z,\mathbb Z)$ and guarantees the existence of both adjoints for each of them. Choose the adjunction $k_s^*h_s\dashv h_{s*}k_s$ on the domain while $k_ih_i^*\dashv h_ik_{i*}$ on the codomain to witness the required taxotopy relation.
\end{rmk}

Now we show that for any $P\in\Pos$ the poset $L(P)$ depends on the chains of $P$, i.e., $L(P)$ is internal to $P$.
\begin{cor}\label{Lisinternal}
Suppose $h,k\in\Pos(\mathbb Z,P)$ satisfy $\im h=\im k$. Then $h\approx k$.
\end{cor}
\begin{proof}
If $\im h=\im k$ for $h,k\in\Pos(\mathbb Z,P)$, then we will only show that $h\preceq k$; taxotopy equivalence will follow by symmetry.

Choose the identity Galois connection on $P$. It only remains to identify the adjunction on $\mathbb Z$. As we observed in the proof of the above theorem, there are five different possibilities based on the order structure of the chain $\im h=\im k$. The adjunctions we constructed on the domain in five different cases work in this case as well. Verification of the details is left to the reader.
\end{proof}

This corollary shows that the speed with which a chain is traced is irrelevant as far as taxotopy is concerned; analogous statement in homotopy theory states that the speed with which a path is traced is irrelevant for homotopy.

If $P$ is a finite poset, then every chain in $P$ is finite. Let $d:=d(P)$ denote the maximum number of distinct elements which can appear in a chain in $P$. Call it the \textbf{height} of $P$. We have seen in Remark \ref{lambdaembeds} that $\lambda(P)$ embeds into $L(P)$. The following result simplifies our job of checking taxotopy relations by converting the domain into a finite total order.
\begin{pro}\label{inftofin}
Suppose $P$ is a finite poset with height $d$.
\begin{itemize}
    \item There are no taxotopy relations between constant and non-constant chains in $P$.
    \item Suppose $h,k:\mathbb Z\to P$ are non-constant maps such that $k\preceq h$. Then there are $h',k'\in\Pos(\mathbf d,P)$ with $\im h=\im{h'},\im k=\im{k'}$ such that $k'\preceq h'$ in $\Lambda(\mathbf d,P)$ where the latter taxotopy relation is witnessed by an adjunction on $\mathbf d$ where both adjoints restrict to identity on $\{0,d-1\}$.
\end{itemize}
\end{pro}
\begin{proof}
The first statement is already proven for maps in $\Pos(\mathbb Z,\mathbb Z)$ with finite images in the proof of Theorem \ref{LZ}. The same argument works in this case.

For the second statement, let $h,k\in\Pos(\mathbb Z,P)$ be non-constant. We use the freedom given by Corollary \ref{Lisinternal} to modify the domain of $h$ to a canonical form keeping the image same. Let $p_0<p_1<\cdots<p_{n-1}$ be the chain given by $\im h$, where $2\leq n\leq d$. Then define the function $h':\mathbb Z\to P$ as follows: $(-\infty,0]\mapsto p_0$, $[d-1,\infty)\mapsto p_{n-1}$, $i\mapsto p_i$ for each $1\leq i\leq n-3$, $[n-2,d-2]\mapsto p_{n-2}$. Similarly, starting from $k$, we obtain the map $k'$ with canonical form. Corollary \ref{Lisinternal} implies that $k'\preceq h'$.

Suppose $f\in\Adj{\mathbb Z}$ and $g\in\Adj P$ witness $k'\preceq h'$. We claim that $f^*(0)\leq 0$. We know from Lemma \ref{adjZ} that $f^*(n)\leq 0$ for infinitely many values of $n$. Let $n_0<0$ be one such value. Now $k'f^*(0)=g^*h'(0)=g^*h'(n_0)=k'f^*(n_0)=k'(0)$. Hence the claim. Similarly we can show that $f_*(0)\leq 0$, $f^*(d-1)\geq d-1$ and $f_*(d-1)\geq d-1$.

Define the map $f'^*\in\Pos(\mathbb Z,P)$ by $f'^*\mathord{\mid}_{(-\infty,0]\cup[d-1,\infty)}=\id{(-\infty,0]\cup[d-1,\infty)}$, $f'^*(n)=f^*(n)$ for $n\in (f^*)^{-1}(0,d-1)$, $f'^*(n)=0$ for $0<n<d-1$ if $f^*(n)\leq 0$ and $f'^*(n)=d-1$ for $0<n<d-1$ if $f^*(n)\geq d-1$. There is a unique right adjoint $f'_*$ for $f'^*$ and it can be checked easily that if $0<f_*(n)<d-1$ then $f'_*(n)=f_*(n)$. We can replace the Galois connection $f$ by $f'$ in the taxotopy square for $k'\preceq h'$ maintaining the commutativity of the taxotopy square.

Finally observe that the maps $h',k'$ restricted to domain $\mathbf d$ have the same images as $h',k'$ respectively. Also the Galois connection $f'$ restricts to a Galois connection on $\mathbf d$ where both restricted adjoints further restrict to the identity map on $\{0,d-1\}$.
\end{proof}

\section{The cone construction}\label{cone}
In the homotopy theory of topological spaces, a contractible space provides an example of a space with ``trivial'' homotopy. Given any topological space $X$, the construction of the cone, $Cone(X)$, is a canonical way of embedding $X$ into a contractible space.

In this section we provide a large class of examples of posets, namely bounded posets, with simple fundamental posets. We will show in Section \ref{taxpos} all bounded posets are null-taxotopic, i.e., they are ``trivial'' with respect to taxotopy.

Following topological intuition we define the cone over a poset.
\begin{definition}
Suppose $P\in\Pos$. The \textbf{cone} over $P$ is a poset obtained from $P$ by adjoining a top and a bottom element. More precisely, we choose two elements $\top,\bot\notin P$. The cone $CP$ has $P\cup\{\top,\bot\}$ as the underlying set, admits an order embedding $P\rightarrowtail CP$, and has relations $\bot<p<\top$ for each $p\in P$.
\end{definition}

\begin{lemma}
Let $a_0<a_1<\cdots<a_{d-1}=\top$ and $\bot=b_0<b_1<\cdots<b_{d-1}$ be two chains in $CP$, where $d\geq 2$. Then there is $f\in\Adj{CP}$ such that for each $0\leq i\leq d-1$ we have $f\models b_i\preceq a_i$.
\end{lemma}

\begin{proof}
Since $a_{d-1}=\top$, the following assignment defines a total monotone function.
\begin{equation*}
f^*(p):=\begin{cases}b_0&\mbox{ if }p\in a_0^\downarrow,\\b_i&\mbox{ if }p\in(a_i^\downarrow\setminus a_{i-1}^\downarrow)\mbox{ for some }i>0.\end{cases}
\end{equation*}

Similarly since $b_0=\bot$, the following assignment also defines a total monotone function.
\begin{equation*}
f_*(p):=\begin{cases}a_{d-1}&\mbox{ if }p\in b_{d-1}^\uparrow,\\a_i&\mbox{ if }p\in(b_i^\uparrow\setminus b_{i+1}^\uparrow)\mbox{ for some }i<d-1.\end{cases}
\end{equation*}

It is easily verified that $f^*\dashv f_*$.
\end{proof}

\begin{theorem}\label{LCone}
Let $P$ be a finite poset with height at least $2$. Then $\lambda(CP)=\mathbf 3$ and $L(CP)=\mathord{\Diamond}\sqcup\mathbf 3$.
\end{theorem}

\begin{proof}
We have already seen in Proposition \ref{inftofin} that elements from $\lambda(CP)$ are not taxotopy related to elements from $L(CP)\setminus\lambda(CP)$.

We first compute $\lambda(CP)$; equivalently, we find taxotopy order between constant maps with values $p$ and $q$. We only need to produce an adjunction on $CP$ depending on whether $p$ and $q$ are elements of $P$ or not.

If $p,q\in P$, then the above lemma applied to the chains $\bot<p<\top$ and $\bot<q<\top$ gives the relation $q\preceq p$ in $\lambda(CP)$. If $p=\top$ and $q\in P$, then the chains $\bot<\top$ and $\bot<q$ give the relation $q\preceq\top$. Dually if $q=\bot$ and $p\in P$, then we obtain $\bot\preceq p$ by considering the chains $p<\top$ and $\bot<\top$.

Finally since a left adjoint preserves $\bot$ and a right adjoint preserves $\top$, we obtain $\top\preceq p \Rightarrow p=\top$ and $q\preceq\bot \Rightarrow q=\bot$. This finishes the proof that $\lambda(CP)=\mathbf 3$ with classes $\bot\prec p\prec\top$ for each $p\in P$.

Now we compute $L(CP)\setminus\lambda(CP)$. We classify all non-constant maps $h:\mathbb Z\to CP$ into four distinct classes depending on the intersection $\im h\cap\{\bot,\top\}$, say $I_\emptyset,I_\bot,I_\top,I_{\bot\top}$. We will show that any two maps in the same class are taxotopy equivalent and, furthermore, the only taxotopy relations are $I_\bot\prec I_{\bot\top}\prec I_\top$ and $I_\bot\prec I_\emptyset\prec I_\top$.

Suppose $h,k$ are non-constant. Proposition \ref{inftofin} tells us that we can think of $h,k$ as maps from $\mathbf d$ to $CP$, where $d$ is the height of $CP$. Since $d\geq 4$, we get $I_\emptyset\neq\emptyset$.

We define and fix an adjunction $e$ on $\mathbf d$ as follows. Since $d\geq4$, the total order $\mathbf d$ can be thought of as a cone over some poset. Applying the above lemma in the case both chains are $0<d-1$ will produce the necessary adjunction. Observe that this adjunction is identity when restricted to $\{0,d-1\}$ and thus is compatible with taxotopy relations in $L(CP)$ in the sense of Proposition \ref{inftofin}.

Let $a_i:=h(i),b_i:=k(i)$ for $i\in\mathbf d$. Then $(a_i)_i,(b_i)_i$ are (non-necessarily strict) chains in $CP$. Since $h,k$ are non-constant, we know that $a_0<a_{d-1}$ and $b_0<b_{d-1}$. It remains to produce adjunction on $CP$ in different cases. We accomplish this by producing chains satisfying the hypothesis of the above lemma which constructs the adjunction.

The chains we are require are strict subchains of $\bot\leq a_0<a_{d-1}\leq\top$ and $\bot\leq b_0<b_{d-1}\leq\top$ of same length, say $\alpha$ and $\beta$ respectively, satisfying the following properties.
\begin{itemize}
\item $\top\in\alpha$, $\bot\in\beta$.
\item The $i^{th}$ entry in $\alpha$ is $a_0$ (resp., $a_{d-1}$) if and only if the $i^{th}$ entry in $\beta$ is $b_0$ (resp. $b_{d-1}$).
\end{itemize}
As an example, suppose $k\in I_\bot$ and $h\in I_\emptyset$, then the chains we use to show $k\preceq h$ are $a_0<a_{d-1}<\top$ and $\bot=b_0<b_{d-1}<\top$. In this way, we can obtain all the nine relations (including reflexivity) to form the diamond.

Now we show that these relations are strict. Suppose $k\preceq h$ for any two non-constant maps $h,k$. The (co)limit preserving properties of adjoints give $\top\in\im k \Rightarrow \top\in\im h$ and $\bot\in\im h \Rightarrow \bot\in\im k$. This completes the proof that $L(CP)=\mathord{\Diamond}\sqcup\mathbf 3$.
\end{proof}

The above theorem does not deal with the cases when the height of the finite poset $P$ is either $0$ or $1$. In the former case $P=\emptyset$ and $CP=\mathbf 2$. It can be checked that $\lambda(\mathbf 2)=\mathbf 2$ and $L(\mathbf 2)=\mathbf 1\sqcup\mathbf 2$. If the height of $P$ is $1$, then $P$ is a finite antichain. In that case, $\lambda(CP)=\mathbf 3$ and $L(CP)=\mathbf 3\sqcup\mathbf 3$.

The following theorem can be proved using very similar methods; the proof is omitted.
\begin{theorem}
If $P\neq\emptyset$ is finite, then $\Lambda(CP)=\mathord{\Diamond}$.
\end{theorem}

As corollary to above theorems, we record the fundamental posets of total orders.
\begin{cor}\label{lambdatotalorder}
For $n\geq 4$, we have $\lambda(\mathbf n)=\mathbf 3$, $\Lambda(\mathbf n)=\mathord{\Diamond}$, and $L(\mathbf n)=\mathord{\Diamond}\sqcup\mathbf 3$. If $P$ is an infinite total order with $m$ endpoints, then $\lambda(P)=\mathbf d$ where $d=m+1$.
\end{cor}

\begin{proof}
Only the second statement needs to be proven. Let $P$ be an infinite total order with $m$ endpoints.

If $p,q\in P$ are not endpoints, then there are $a,b\in P$ such that both $p,q$ belong to the open interval $(a,b)$. The closed interval $[a,b]$ is the cone over the open interval $(a,b)$. Moreover every adjunction on $[a,b]$ can be extended to an adjunction on $P$ using identity outside $[a,b]$. The first part of the proof of Theorem \ref{LCone} can be applied to $[a,b]$ to show that $p$ and $q$ are taxotopy equivalent. The minimum/maximum of the poset, if exists, is also the minimum/maximum of the taxotopy order for adjoints preserve (co)limits. This completes the proof.
\end{proof}

\begin{rmk}
It should be noted that, in general, the poset $\lambda(P)$ depends on the whole of $P$. The knowledge of a part of $P$ may not give any information on a part of $\lambda(P)$. In particular, the relations $p\leq q$ and $p\preceq q$ are independent unless either of the elements has a special designation in the whole of $P$ (e.g., when $q$ is the top element). In this respect, we can say that $\lambda(P)$ is ``orthogonal'' to $P$.
\end{rmk}

\section{$\Lambda(S,-)$ as a functor}
Given a monotone map $h:Q\to P$ between two posets, it is natural to ask under what conditions on $h$ we get an induced map $\lambda(h):\lambda(Q)\to\lambda(P)$. In other words we ask for a subcategory $\mathcal C$ of $\Pos$ where $\lambda:\mathcal C\to\Pos$ is a functor.

The answer is very simple if we refer to the adjoint topology on a poset! A map $h:Q\to P$ induces a map $\lambda(Q)\to\lambda(P)$ if and only if $h$ is continuous with respect to the adjoint topologies on $Q$ and $P$. This statement is true irrespective of whether $h$ is actually monotone but we will only restrict our attention to monotone continuous maps. The collection $\Cont{}$ of all monotone continuous maps is a subcategory of $\Pos$.

In other words, $\lambda:\Cont{}\to\Pos$ is a functor, i.e., if $q\preceq q'$ in $\lambda(Q)$ and $h\in\Cont{}(Q,P)$, then $h(q)\preceq h(q')$ in $\lambda(P)$. We can treat $q,q'$ as maps in $\Pos(\mathbf 1,Q)$. This motivates the following definition.
\begin{definition}
Let $\emptyset\neq S\in\Pos$. A map $h\in\Pos(Q,P)$ is said to be $S$-\textbf{continuous} if whenever $k\preceq k'$ in $\Lambda(S,Q)$ then $h\circ k\preceq h\circ k'$ in $\Lambda(S,P)$. If we denote the subcategory of $\Pos$ consisting of $S$-continuous maps as $\Cont S$, then $\Lambda(S,-):\Cont S\to\Pos$ is a functor.
\end{definition}
Note that $\Cont{}=\Cont{\mathbf 1}$ and $\Cont S\subseteq\Cont{}$ for each nonempty $S$, for a taxotopy relation between two constant maps depends only on the existence of an adjunction on the codomain.

$S$-continuity is a statement that guarantees the existence of two adjunctions given two specific adjunctions, i.e., if $(e,g)\models k\preceq k'$, then there exist $e'\in\Adj S,\ f'\in\Adj P$ such that $(e',f')\models h\circ k\preceq h\circ k'$. But the adjunction $e'$ may not be related to $e$. Similarly, we may not expect (bi)commutativity of the square involving only $Q$ and $P$. This makes it hard to work with $S$-continuous maps themselves unless some extra structural information about the posets is known. The following definition introduces a class of maps where the existence clause is fairly simple.
\begin{definition}
A monotone map $h:Q\to P$ is said have the \textbf{extension property} if each $f\in\Adj Q$ yields an adjunction $f^{ext}$ on $P$ such that $(f,f^{ext})\models h\preceq h$. We denote by $\mathcal E$ the subcategory of $\Pos$ consisting of monotone maps with extension property.
\end{definition}
Note that $\mathcal E\subseteq\Cont S$ for all $S\neq\emptyset$.
\begin{examples}
The following have the extension property: the unique map $P\to\mathbf 1$ for each nonempty $P$; any constant map; inclusion of any subset of $\mathbb Z$; inclusion of any connected component of a poset; inclusion of an upper subset $U$ of $P$ if each element of $U$ is strictly above each element of $P\setminus U$; any embedding $\mathbf 3\to CP$ that preserves $\top$ and $\bot$ for nonempty $P$; the unique surjection $CP\to\mathbf 3$ that maps $P$ onto $1\in\mathbf 3$ for nonempty $P$.
\end{examples}

So far we discussed preservation of the taxotopy order. Now we define a class of maps that reflect taxotopy order.
\begin{definitions}
An order embedding $i:Q\rightarrowtail P$ is said to have the \textbf{restriction property} if each $f\in\Adj P$ restricts to an adjunction $f\mathord{\mid}_Q$ on $Q$, i.e., $(f\mathord{\mid}_Q,f)\models i\preceq i$. We denote by $\mathcal R$ the subcategory of $\Pos$ consisting of order embeddings with restriction property.

We denote by $\Subw P$ (resp. $\Subs P$) the collection of all subsets $Q\subseteq P$ such that the inclusion $Q\rightarrowtail P$ is in $\mathcal R$ (resp. in $\mathcal R\cap\mathcal E$).
\end{definitions}
As noted above, an order embedding $i:Q\rightarrowtail P$ with restriction property reflects taxotopy order, i.e., for $k,k'$ in $\Pos(S,Q)$, if $i\circ k\preceq i\circ k'$ in $\Lambda(S,P)$ then $k\preceq k'$ in $\Lambda(S,Q)$.

Suppose $Q'\subseteq Q\subseteq P$. If $Q,Q'\in\Subw P$, then $Q'$ may not belong to $\Subw Q$. But if $Q,Q'\in\Subs P$, then $Q'\in\Subs Q$.

\begin{rmk}
If $\emptyset\neq Q\in\Subs P$, then $f\mapsto f\mathord{\mid}_Q:\Adj P\to\Adj Q$ is a left inverse for any map $g\mapsto g^{ext}:\Adj Q\to\Adj P$. As a consequence, the inclusion induces an order embedding $\Lambda(S,Q)\rightarrowtail\Lambda(S,P)$ for each $S\in\Pos$.

If $P\neq\emptyset$, then the poset $(\Subw P,\subseteq)$ is complete but $(\Subs P,\subseteq)$ need not even be closed under finite meets or joins.
\end{rmk}

\section{Taxotopy of posets}\label{taxpos}
Two topological spaces $X$ and $Y$ are homotopy equivalent if there are continuous maps $f:X\rightleftarrows Y:g$ such that $gf\simeq\id X$ and $fg\simeq\id Y$. Consistent with the theme of this paper, instead of defining taxotopy equivalence relation on the posets directly, our aim now is to define taxotopy preorder between posets using a similar idea.
\begin{definition}
Suppose $S$ is a nonempty poset. A \textbf{weak $S$-adjoint} is a pair of $S$-continuous maps $h^+:Q\rightleftarrows P:h_+$ such that $\id Q\preceq h_+h^+$ in $\Pos(\Lambda(S,Q),\Lambda(S,Q))$ and $h^+h_+\preceq\id P$ in $\Pos(\Lambda(S,P),\Lambda(S,P))$, i.e., $k'\preceq h_+h^+k'$ for each $k'\in\Pos(S,Q)$ and $h^+h_+k\preceq k$ for each $k\in\Pos(S,P)$. We write this as $h^+\simv S h_+$.
\end{definition}
As usual the notations $h^+,h_+$ denote the left and right weak adjoints respectively. A weak $\mathbf 1$-adjoint will simply be referred to as a weak adjoint.

We use the word adjunction in the above definition because a weak $S$-adjunction $h^+\simv S h_+$ induces an actual adjunction $\Lambda(S,h^+)\dashv\Lambda(S,h_+)$ between the fundamental posets.

\begin{definition}\label{nulltaxotopy}
Fix a nonempty $S$. We say that a poset $P$ is \textbf{taxotopic} (resp. $S$-taxotopic) to $Q$, written $P\preceq Q$ (resp. $P\preceq_S Q$), if there exists a weak adjunction (resp. weak $S$-adjunction) $h^+:P\rightleftarrows Q:h_+$ where the right adjoint goes from $P$ to $Q$.

We say that $P$ and $Q$ are \textbf{taxotopy equivalent} (resp. $S$-taxotopy equivalent), written $P\approx Q$ (resp. $P\approx_S Q$), if $P\preceq Q$ and $Q\preceq P$ (resp. $P\preceq_S Q$ and $Q\preceq_S P$).

We say that a poset is \textbf{null-taxotopic} if it is taxotopy equivalent to $\mathbf 1$.
\end{definition}
It is easy to see that the relations $\preceq_S$ are reflexive and transitive, and hence preorders.

\begin{pro}\label{listnulltax}
Each poset in the following list is null-taxotopic: a total order, the cone $CP$ for any poset $P$ and hence any complete lattice, any arbitrary copower of any poset in this list.
\end{pro}

\begin{proof}
Let $T$ be a total order without endpoints. Any monotone map $\mathbf 1\to T$ has the extension property and so does the unique map $!:T\to\mathbf 1$. These continuous maps induce isomorphisms between fundamental posets since $\lambda(T)=\mathbf 1=\lambda(\mathbf 1)$ and hence $T\approx\mathbf 1$.

For an infinite total $T'$ order with one endpoint, we have seen that $\lambda(T')=\mathbf 2$. The constant map $\mathbf 2\to\mathbf 1$ admits both adjoints. Moreover these adjoints are induced by some continuous maps $\mathbf 1\to T'$ and hence $T'\approx\mathbf 1$.

For a poset $P$, the unique map $!:CP\to\mathbf 1$ admits both left and right weak adjoints, namely the maps with images $\bot$ and $\top$ respectively.

Finally note that $\lambda$ of an arbitrary copower of a poset is $\lambda$ of the poset itself and hence the proof.
\end{proof}

A smallest example of a non-null-taxotopic poset contains three elements. Let $\mathbf V$ denote the poset with three elements and whose Hasse diagrams look like the letter $V$. Recall that $\mathbf V$ is contractible as a topological space but there are no taxotopy relations between $\mathbf 1$ and $\mathbf V$. Moreover if P is any null-taxotopic poset and $\ddot P$ denotes the poset containing exactly two incomparable elements strictly above all elements of P, then $\mathbf V\approx\ddot P$. In summary, the structure of the poset plays an important role in determining its taxotopy equivalence class.

Now we proceed to find a ``taxotopy reduct'' of a poset, i.e., a subset of a poset, produced in a canonical way, that is taxotopy equivalent to the original poset.
\begin{definition}
Let $P$ be a non-empty poset. Given $a<b$ in $P$, the closed interval $[a,b]$ is said to be a \textbf{tunnel} if the open interval $(a,b)$ contains at least two elements and for each $c\in(a,b)$ and $d\notin[a,b]$, if $c<d$ then $b<d$ and if $c>d$ then $a>d$.
\end{definition}
If the interiors of two tunnels intersect then one of them is contained in the other. We can collapse a tunnel pair to a copy of $\mathbf 3$ without affecting the fundamental poset.
\begin{pro}
Let $P$ be a poset with a tunnel $[a,b]$. If $P^1$ denotes the poset in which the open interval $(a,b)$ is replaced by a single element $c$ satisfying $a<c<b$, then $\lambda(P^1)=\lambda(P)$.
\end{pro}

\begin{proof}
Since the tunnel $[a,b]$ is the cone over the open interval $(a,b)$, the class of each element in $(a,b)$ in $\lambda(P)$ is the same. Using the tunnel condition, the unique map $(a,b)\to\{c\}$ and any embedding of $\{c\}$ into $(a,b)$ can be extended to monotone maps between $P$ and $P^1$ using identity on the complements. These extensions are continuous and form a pair weak adjoints.
\end{proof}

If $P$ is a finite poset, then there is a finite number of maximal tunnels which could be reduced using the above procedure to obtain a tunnel-free or a ``reduced'' poset. It is simpler to apply the techniques in the following sections for computing fundamental posets to such reduced posets.

\section{Fundamental posets of disjoint unions}
One of the important goals of this article is to compute the fundamental poset of a given poset using the fundamental posets of its parts. The case of a connected poset will be dealt with in later sections but now we concentrate on the posets with more than one connected component.

To state our results we need the definition of the generalized taxotopy relation.
\begin{definition}
Suppose $P$ and $Q$ are connected posets. For $p\in P,q\in Q$, we say that $p$ is \textbf{taxotopic} to $q$, written $p\preceq_P^Qq$, if there is an adjunction $f^*:Q\rightleftarrows P:f_*$ such that $f^*(q)=p$ and $f_*(p)=q$.

More generally, for a fixed non-empty poset $S$, we say that $k:S\to P$ is taxotopic to $h:S\to Q$, written $k{}_S\preceq_P^Q h$, if there are adjunctions $e^*\dashv e_*$ on $S$ and $f^*:Q\rightleftarrows P:f_*$ such that $f^*\circ h=k\circ e^*$ and $f_*\circ k=h\circ e_*$.
\end{definition}
Note that the former case obtained when $S=1$. We can recover the fundamental poset $\Lambda(S,P)$ by setting $Q=P$ in the latter definition.

The following observation is crucial.
\begin{pro}
Let $P=\bigsqcup_{i\in I}P_i$ be the decomposition of a poset $P$ into its connected components. Any adjunction $f$ on $P$ induces a permutation $\phi:I\to I$ such that $f_*(P_i)\subseteq P_{\phi(i)}$.
\end{pro}
A permutation on the index set, $I$, that can be obtained from an adjunction in this way will be termed an admissible permutation. An ordered pair $(i,j)$ will be termed admissible, written $i\preceq j$, if there exists an admissible permutation $\phi$ on $I$ such that $j=\phi(i)$. It is clear that the admissibility relation, $\preceq$, on the ordered pairs is indeed a preorder. Furthermore if $i\preceq j$, $j\preceq k$, $p_i\preceq_{P_i}^{P_j} p_j$ and $p_j\preceq_{P_j}^{P_k} p_k$, then $p_i\preceq_{P_i}^{P_k}p_k$.

\begin{pro}\label{lambdadisjoint}
Let $S$ be a connected poset and $\{P_i\}_{i\in I}$ be a family of non-empty connected posets. Then
\begin{equation*}
\Lambda(S,\bigsqcup_{i\in I}P_i)=\left(\bigsqcup_{i\in I}\Pos(S,P_i),\bigsqcup_{i\preceq j}{}_S\preceq_{P_i}^{P_j}\right)/\left(\bigsqcup_{i\approx j}{}_S\preceq_{P_i}^{P_j}\cap{}_S\preceq_{P_j}^{P_i}\right).
\end{equation*}
\end{pro}

\begin{proof}
The image of a monotone map with connected domain is contained in a connected component of the codomain. Therefore $k\preceq h$ in $\Lambda(S,\bigsqcup_{i\in I}P_i)$ if and only if $k{}_S\preceq_{P_i}^{P_j}h$ for some $i\preceq j$.
\end{proof}

When the domain has more than one connected component, the situation is more complicated.
\begin{definitions}
\textbf{The folding quotient:} Given $Q\in\Set$ and a poset $(Q^2,\leq)$ with \textbf{coordinate symmetry}, i.e., $(q_1,q_2)\leq(q'_1,q'_2)$ if and only if $(q_2,q_1)\leq(q'_2,q'_1)$, the folding quotient $\mbox{Fold}(Q^2,\leq)$ is obtained by identifying the ordered pairs $(q_1,q_2)$ and $(q_2,q_1)$.

\textbf{The open book condition:} Let $P$ be a non-empty connected poset. We say that the ordered pairs $(h_1,h_2),(k_1,k_2)\in\Lambda(P)\times\Lambda(P)$ are linked by an open book condition if there exist $f,g,e\in\Adj P$ such that $(f,e)\models h_1\preceq k_1$ and $(g,e)\models h_2\preceq k_2$. We write this as $(h_1,h_2)\preceq^o(k_1,k_2)$. Note that $\preceq^o$ satisfies coordinate symmetry.
\end{definitions}

\begin{pro}
Let $P$ be a connected poset. Then
\begin{equation*}
\Lambda(P\sqcup P)=\mbox{Fold}(\Lambda(P)^2,\preceq)\sqcup\mbox{Fold}(\Lambda(P)^2,\preceq^o).
\end{equation*}
\end{pro}

\begin{proof}
Since the image of a connected component of a poset under a monotone map is contained in a connected component of the codomain, we can think of a map $h=h_1\sqcup h_2\in\Pos(P\sqcup P,P\sqcup P)$ as a pair of maps $h_1,h_2\in\Pos(P,P)$ by corestriction. We can classify maps in $\Pos(P\sqcup P,P\sqcup P)$ into two classes depending on whether the images of $h_1$ and $h_2$ are included in the same connected component of the codomain. It is easy to see that there are no taxotopy relations between maps from distinct classes.

For simplicity, we name the components of $P\sqcup P$ as $P_1$ and $P_2$. Let $r$ be the unique involution which is the coproduct of $\id P:P_1\to P_2$ and $\id P:P_2\to P_1$.

Case I: Suppose $h=h_1\sqcup h_2$ and $k=k_1\sqcup k_2$ with $h_1\in\Pos(P_1,P_1)$ and $h_2,k_2\in\Pos(P_2,P_2)$. Then $h\preceq k$ in $\Lambda(P\sqcup P)$ if and only if $h_1\preceq k_1$ and $h_2\preceq k_2$ in $\Lambda(P,P)$. Finally the folding of the poset $\Lambda(P)^2$ comes from the fact that $(h_1\sqcup h_2)\approx(h_1\sqcup h_2)\circ r=(h_2\sqcup h_1)$.

Case II: Suppose $\im h,\im k\subseteq P$. Then $h\preceq k$ in $\Lambda(P\sqcup P)$ if an only if $h\preceq k$ in $\Lambda(P\sqcup P,P)$. One can think of a monotone map in $\Pos(P\sqcup P,P)$ as a pair of maps in $\Pos(P,P)$. But the taxotopy order in $\Lambda(P\sqcup P,P)$ is not the full product order in $\Lambda(P)^2$. In fact it is isomorphic to $\mbox{Fold}(\Lambda(P)^2,\preceq^o)$ where $\preceq^o$ is the order obtained from the open book condition.
\end{proof}

\begin{example}
We illustrate the above result by computing $\Lambda(P\sqcup P)$ where $P=\mathbf 2$. We know that $\Lambda(P)=\mathbf 3$. The folding quotient of $\Lambda(P)^2$ with usual taxotopy order is isomorphic to $C\mathord{\Diamond}$ where $C$ denotes the cone operation. Only the unique non-identity adjunction on $\mathbf 2$ can be used as the middle adjunction $e$ in an open book diagram. Computations reveal that the Hasse diagram for the order $\preceq^o$ on $\mathbf 3\times\mathbf 3$ can be obtained by removing the outer edges of the Hasse diagram for the lattice $\mathbf 3\times\mathbf 3$ and adding in two edges for $(0,0)\prec^o(1,1)\prec^o(2,2)$. The folding quotient of this poset looks like the coproduct $\mathbf X\sqcup\mathbf 1$, where $\mathbf X$ denotes the poset with five points whose Hasse diagram looks like the letter $X$. Therefore $\Lambda(P\sqcup P)=C\mathord{\Diamond}\sqcup\mathbf X\sqcup\mathbf 1$.
\end{example}

The above proposition could be thought of as the special case of $P=P_1\sqcup P_2$ where any ordered pair from the index set $\{1,2\}$ is admissible. The following result deals with the other extreme case.
\begin{pro}
Let $P_1,P_2$ be two connected posets and $P=P_1\sqcup P_2$. If the admissibility preorder on the index set $\{1,2\}$ is the identity relation, then
\begin{equation*}
\Lambda(P)=\left[\Lambda(P_1)\times\Lambda(P_2)\sqcup\Lambda(P_1,P_2)\times\Lambda(P_2,P_1)\right]\sqcup\left[\Lambda(P_1\sqcup P_2,P_1)\sqcup\Lambda(P_1\sqcup P_2,P_2)\right].
\end{equation*}
\end{pro}

\section{Seifert-van Kampen theorem for $\lambda$}
The Seifert-van Kampen theorem for fundamental groups in the homotopy theory of topological spaces states that if a topological space $X$ can be written as the union of two path connected open subspaces $U_1,U_2$ such that $U_1\cap U_2$ is also path connected, then, for $x_0\in U_1\cap U_2$, we have the pushout diagram $\pi_1(X;x_0)=\pi_1(U_;x_0)\ast_{\pi_1(U_1\cap U_2;x_0)}\pi_1(U_2;x_0)$ in the category of groups. The following result provides its analogue in the taxotopy theory of posets.

\begin{theorem}[Seifert-van Kampen theorem for $\lambda$]\label{vanKampenlambda}
If $P_1,P_2$ are subsets of a poset $P$ such that all the order embeddings in the pushout diagram $P=P_1\cup_{P_1\cap P_2}P_2$ are continuous and have the restriction property, then there is a pushout diagram
\begin{equation*}
\lambda(P)=\lambda(P_1)\cup_{\lambda(P_1\cap P_2)}\lambda(P_2)
\end{equation*}
of order embeddings.
\end{theorem}

\begin{proof}
Since all order embeddings in the pushout diagram $P=P_1\cup_{P_1\cap P_2}P_2$ are continuous and have the restriction property, we get induced order embeddings by applying $\lambda(-)$.

It only remains to show that $\lambda(P)$ is indeed the pushout of the diagram $\lambda(P_1)\leftarrowtail\lambda(P_1\cap P_2)\rightarrowtail\lambda(P_2)$. Let $\lambda(P_1)\xrightarrow{h_1} R\xleftarrow{h_2}\lambda(P_2)$ be any cocone over this diagram in $\Pos$. We define $h:\lambda(P)\to R$ by $h(p):=\begin{cases}h_1(p)&\mbox{if }p\in P_1\\h_2(p)&\mbox{otherwise}\end{cases}$. Since $h_1(p)=h_2(p)$ for $p\in P_1\cap P_2$, we get that $h$ is well-defined. It is clear that $h_1$ and $h_2$ factor through $h$ and $h$ is the unique map with this property.
\end{proof}

One cannot expect to obtain analogous result when $\lambda(-)$ is replaced by $\Lambda(S,-)$ with $S\neq\mathbf 1$ for the image of $S$ under a monotone map may intersect both $P_1$ and $P_2$ but may not be contained in $P_1\cap P_2$. We will see in Corollary \ref{vanKampenLambdastandard} that one needs to reverse the arrows in the original diagram to get a similar result for $\Lambda(S,-)$ under different hypotheses.

\begin{rmk}
The restriction hypothesis on the inclusions in the above theorem is necessary. For a counterexample consider $P\sqcup P$ expressed as the pushout $P\cup_\emptyset P$. Proposition \ref{lambdadisjoint} gives that $\lambda(P\sqcup P)=\lambda(P)$ but $\lambda(P)\neq\lambda(P)\cup_{\lambda(\emptyset)}\lambda(P)$! This happens because the inclusions $P\rightarrowtail P\sqcup P$ do not have the restriction property.

In fact we shall give up continuity in the next sections but retain the restriction property to obtain a van Kampen theorem for $\Lambda(S,-)$ when $S\neq\mathbf 1$.
\end{rmk}

\section{Covering a poset}
Our topological intuition says that one could compute $\Lambda(S,-)$ of a poset out of $\Lambda(S,-)$ of its subsets provided those subsets cover the original poset ``nicely''. This section introduces the notion of a cover of a poset that allows us to construct an adjunction out of its restrictions at all elements of the cover.
\begin{definition}
Let $P$ be a nonempty poset. A collection $\mathcal Q$ of subsets of $P$ is said to be a \textbf{cover} if it satisfies the following properties.
\begin{itemize}
\item $\emptyset\notin\mathcal Q\subseteq\Subw P$.
\item $\mathcal Q$ is closed under nonempty finite intersections and $\bigvee\mathcal Q=P$.
\item For each $Q'\subseteq Q$ in $\mathcal Q$, the inclusion $Q'\rightarrowtail Q$ is in $\mathcal R$.
\end{itemize}
\end{definition}

For the most general form of the van Kampen theorem, Theorem \ref{vanKampengen}, we need a stronger form of a cover.
\begin{definition}
We say that a cover $\mathcal Q$ of a poset $P$ is \textbf{chain-compact} if for every maximal chain $I$ in $P$ (existence by Zorn's lemma) there exists a finite subset $\mathcal Q'$ of $\mathcal Q$ such that
\begin{itemize}
\item $I\subseteq\bigcup\mathcal Q'$,
\item for each $a<b$ in $Q\cap I$, $[a,b]\cap I\subseteq Q$,
\item for each $a<b$ in $I$, if there does not exist any $Q\in\mathcal Q$ with $a,b\in Q$, then there is a finite subchain $a=p_0<p_1<\hdots<p_l=b$ of $I$ such that for each $1\leq i\leq l$, there is some $Q\in\mathcal Q'$ such that $[p_{i-1},p_i]\cap I\subseteq Q$.
\end{itemize}
\end{definition}
Unlike the usual use of the word ``compact'', here we attribute this adjective to a cover rather than the poset.

The third condition in this definition says that all $\leq$-signs are finitely covered and therefore there could be finite covers of posets which are not chain-compact! For example, consider the poset consisting of exactly four points and whose Hasse diagram looks like $\Join$. Two subsets consisting of the maximal and minimal points form its cover. But this finite cover is not chain-compact since the inclusions of minimal elements in the maximal elements are not contained in any of the covers.

\begin{definition}
Let $\mathcal Q$ be a cover of $P$.

A family of maps $(k_Q:Q\to Q)_{Q\in\mathcal Q}$ is said to be \textbf{compatible} if whenever $Q'\subseteq Q$ in $\mathcal Q$ we have $k_{Q'}(q')=k_Q(q')$ for each $q'\in Q'$.

A family of adjunctions $(f_Q\in\Adj Q)_{Q\in\mathcal Q}$ is said to be \textbf{compatible} if whenever $Q'\subseteq Q$ in $\mathcal Q$ we have $(f_Q)\mathord{\mid}_{Q'}=f_{Q'}$.

A cover $\mathcal Q$ of $P$ is said to have the \textbf{extension property} if for every compatible family $\{f_Q\in\Adj Q:Q\in\mathcal Q\}$ of adjunctions, there exists $f\in\Adj P$ with $f\mathord{\mid}_Q=f_Q$ for each $Q\in\mathcal Q$.
\end{definition}

\begin{pro}\label{compactextension}
A chain-compact cover $\mathcal Q$ of a poset $P$ has the extension property.
\end{pro}
\begin{proof}
Let $\mathcal Q$ be a chain-compact cover of $P$. We show that whenever $(k_Q:Q\to Q)_{Q\in\mathcal Q}$ is a compatible family of monotone maps, the function $k:P\to P$ defined by $k(p):=k_Q(p)$ for $p\in Q$ is monotone.

Suppose $p<p'$ in $P$. We prove that $k(p)\leq k(p')$.

If $p,p'\in Q$ for some $Q\in\mathcal Q$, we are done using monotonicity of $k_Q$.

Otherwise, assuming Zorn's lemma, there is a maximal chain $I$ in $P$ with $p,p'\in I$. Since $\mathcal Q$ is chain-compact, there is a finite subset $\mathcal Q'$ of $\mathcal Q$ and a finite subchain $p=p_0<p_1<\hdots<p_l=b$ of $I$ such that for each $1\leq i\leq l$, there is some $Q_i\in\mathcal Q'$ such that $[p_{i-1},p_i]\cap I\subseteq Q_i$. So we obtain $k(p)=k_{Q_1}(p_0)\leq k_{Q_1}(p_1)=k_{Q_2}(p_1)\leq k_{Q_2}(p_2)=k_{Q_3}(p_2)\leq\hdots\leq k_{Q_l}(p_l)=k(p')$.

Since left and right adjoints from a compatible family of adjunctions form compatible families of monotone maps, by above process we can obtain two monotone maps in $\Pos(P,P)$ which could be verified to be adjoint to each other.
\end{proof}

The following result will be useful in the next section.
\begin{pro}\label{compactcomparison}
Suppose $S\subseteq\mathbb Z$, $P\in\Pos$ and $h\preceq h'$ in $\Lambda(S,P)$. Then $\im h$ has a maximum element if and only if $\im{h'}$ has a maximum element.
\end{pro}

\begin{proof}
Let $h\preceq h'$ in $\Lambda(S,P)$. Suppose $e\in\Adj S,f\in\Adj P$ witness this relation. Let $I:=\im h\cap\im{f^*}$ and $I':=\im{h'}\cap\im{f_*}$. Then the adjunction $f$ induces an order isomorphism $I\to I'$.

If $S$ has a maximum element, then $h(\max S)$ and $h'(\max S)$ are maximum elements in $\im h$ and $\im{h'}$ respectively and hence there is nothing to be proven.

Suppose that $S$ has no maximum element. It can be easily shown that $h$ is constant if and only if $h'$ is constant. So we assume that both $h,h'$ are non-constant. Note that $I'$ is cofinal in $\im{h'}$. There are two cases.

If $\im{h'}$ does not have a maximum element, then so does $I'$ and hence $I$. If $\im h$ has a maximum element, then $f^*f_*(\max\im h)$ is the largest element of $I$, which does not exist. Hence $\im h$ does not have a maximum element.

If $\im{h'}$ has a maximum element but $\im h$ does not have a maximum element. By assumptions, $(\max I)^\uparrow\cap\im h$ is order isomorphic to $\mathbb N$ and, for each $p$ in this set, we know that $f^*f_*(p)=\max I$.

The set $h'^{-1}(\max I')$ is an upper set in $S$ but $h^{-1}(\max I)$ is not. For any $p>\max I$ in $P$, $e_*(p)$ is an upper bound for the set $h'^{-1}(\max I')$ which is a contradiction. Hence $\im h$ has a maximum element. This completes the proof of both cases and hence of the proposition.
\end{proof}

Duality provides analogue of this result when maximum is replaced by minimum.

\section{Seifert-van Kampen theorem for $\Lambda(S,-)$}
We need a lot of technical machinery to state and prove the main result of this section. For a cover $\mathcal Q$ of a poset $P$, the concepts of a $\mathcal Q$-indexed matching family and the amalgamation property we introduce here are reminiscent of the concepts with the same name in topos theory (cf. \cite[\S III.4]{Moerdijk}) but we repeat them in our context for the convenience of the reader.
\begin{definition}
We say that a nonempty total order $S$ is \textbf{homogeneous} if for any $P\in\Pos$ and $h,k\in\Pos(S,P)$ whenever $\im h=\im k$ then there are $e_1,e_2\in\Adj S$ such that $(e_1,\id P)\models h\preceq k$ and $(e_2,\id P)\models k\preceq h$.
\end{definition}
We have already seen in Corollary \ref{Lisinternal} that $\mathbb Z$ is homogeneous. It can be easily verified that all subsets of $\mathbb Z$ are homogeneous. Whenever we assume $S\subseteq\mathbb Z$ in this section, we actually refer to the fact that $S$ is homogeneous.

\begin{definition}
Given $\emptyset\neq P\in\Pos$, $\emptyset\neq Q\in\Subw P$ and a homogeneous total order $S$, the inclusion map $i:Q\rightarrowtail P$ is said to \textbf{induce taxotopy retracts of $S$-chains} if the following condition is satisfied.

For each $h\in\Pos(S,P)$ form the pullback $S\xleftarrow{j} S\times_PQ\xrightarrow{k}Q$. If $S\times_PQ\neq\emptyset$, then there is a diagonal $\widetilde{h}:S\to Q$ such that $\widetilde{h}\circ j=k$.

If a diagonal $\widetilde{h}$ exists, then it is referred to as a \textbf{taxotopy retract} of the map $h$ with respect to $Q$.
\end{definition}
A taxotopy retract of $h$ exists only if $\im h\cap Q\neq\emptyset$. Though we do not require uniqueness of the taxotopy retract, since $S$ is homogeneous, the diagonal, if exists, is unique up to taxotopy equivalence.

The inclusion $i:Q\rightarrowtail P$ naturally induces a (monotone) map $i\circ(-):\Pos(S,Q)\to\Pos(S,P)$. The above definition provides a (not necessarily monotone) partial function $\widetilde{(-)}: \Pos(S,P)\rightharpoonup\Pos(S,Q)$ such that for each $k\in\Pos(S,Q)$ we have $\widetilde{i\circ k}=k$.

\begin{rmk}
If the taxotopy retract $\widetilde{h}\in\Pos(S,Q)$ of $h\in\Pos(S,P)$ exists, then for each $s\in S$ the element $\widetilde{h}(s)$ must be equal to either $\max\{h(s^\downarrow)\cap Q\}$ or $\min\{h(s^\uparrow)\cap Q\}$.

If $S\subseteq\mathbb Z$, then both these elements exist unless the sets are empty, and we can define $\widetilde{h}:S\to Q$ as follows.
\begin{equation*}
\widetilde{h}(s):=\begin{cases}\max\{h(s^\downarrow)\cap Q\}&\mbox{if }\{h(s^\downarrow)\cap Q\}\neq\emptyset,\\\min\{h(s^\uparrow)\cap Q\}&\mbox{otherwise}.\end{cases}
\end{equation*}
This construction in fact guarantees that whenever $S\subseteq\mathbb Z$, each inclusion of posets induces taxotopy retracts of $S$-chains irrespective of whether it has the restriction property.
\end{rmk}

\begin{pro}\label{retractnormal}
Suppose $S\subseteq\mathbb Z$ and $Q\in\Subw P$. Given $S$-chains $\alpha,\beta$ in $P$, if $\alpha\preceq\beta$ in $\Lambda(S,P)$ then $\alpha\cap Q\neq\emptyset$ if and only if $\beta\cap Q\neq\emptyset$. If both $\alpha\cap Q,\beta\cap Q$ are nonempty, then $(\alpha\cap Q)\preceq(\beta\cap Q)$ in $\Lambda(S,Q)$.
\end{pro}

\begin{proof}
Suppose $h,k$ are maps in $\Pos(S,P)$ with images $\alpha,\beta$ respectively with $\alpha\preceq\beta$. Then there are $f\in\Adj P$ and $e\in\Adj S$ such that $(e,f)\models h\preceq k$, i.e., $(e,f)\models\alpha\preceq\beta$.

Suppose $\alpha\cap Q\neq\emptyset$. Since $Q\in\Subw P$, we know that $f\mathord{\mid}_Q\in\Adj Q$. Then $(f\mathord{\mid}_Q)_*(\alpha\cap Q)\subseteq f_*(\alpha)\cap Q=\beta\cap Q$ and $(f\mathord{\mid}_Q)^*(\beta\cap Q)\subseteq f^*(\beta)\cap Q=\alpha\cap Q$. Since $f^*$ and $f_*$ are bijections when restricted to $\beta$ and $\alpha$ respectively, we get that $(f\mathord{\mid}_Q)_*(\alpha\cap Q)=\beta\cap Q$ and $(f\mathord{\mid}_Q)^*(\beta\cap Q)=\alpha\cap Q$. A similar argument applies if $\beta\cap Q\neq\emptyset$. This proves the first statement.

Suppose $\alpha\cap Q$ and $\beta\cap Q$ are both nonempty. Define the taxotopy retracts $h_Q,k_Q\in\Pos(S,Q)$ of $h,k$ as in the above remark with images $\alpha\cap Q$ and $\beta\cap Q$ respectively. For each $s\in S$, if $h(s)\in Q$ then $h(s)=h_Q(s)$ and if $k(s)\in Q$ then $k(s)=k_Q(s)$. Finally observe that $(e,f\mathord{\mid}_Q)\models h_Q\preceq k_Q$.
\end{proof}

We want to give a uniform definition of the taxotopy retract of each map, and thus we adjoin a bottom element $\bot$ to each poset $P$, which we denote by $P_\bot$.
\begin{definition}
Let $S$ be a homogeneous total order. A family $\{h_Q:S\to Q_\bot\mid Q\in\mathcal Q\}$ of monotone maps indexed by the elements of a cover $\mathcal Q$ of $P$ is a \textbf{matching family} if for each $s\in S$ and $Q'\subseteq Q$ in $\mathcal Q$, if $h_Q(s^\downarrow)\cap Q'$ is nonempty then its maximum exists and  $h_{Q'}(s)$ equals this maximum; otherwise $h_{Q'}(s)=\bot$.
\end{definition}
Note that if $Q'\subseteq Q$ then $h_{Q'}$ is a taxotopy retract of $h_Q$ as intended. In fact the second condition of the definition actually provides a way to construct $h_{Q'}$ from $h_Q$ which could be called the ``greatest lower bound'' construction where $\bot$ is assumed to be the greatest lower bound of the empty set; we denote this assignment by $h_{Q'}=\rho_{Q,Q'}(h_Q)$.

\begin{definition}
A matching family $\{h_Q:S\to Q_\bot\mid Q\in\mathcal Q\}$ indexed by the elements of a cover $\mathcal Q$ of $P$ is said to have the \textbf{amalgamation property} if $\max\{h_Q(s)\mid Q\in\mathcal Q\}$ exists, denoted $p_s$, for each $s\in S$.
\end{definition}

Assuming the amalgamation property for a matching family, we can define the function, called the amalgam $h:S\to P_\bot$ by $h(s):=p_s$. To verify that $h$ is monotone let $s_1\leq s_2$ in $S$. Then there is some $Q\in\mathcal Q$ such that $p_{s_1}=h_Q(s_1)$. Clearly $h_Q(s_1)\leq h_Q(s_2)\leq p_{s_2}$.

\begin{pro}\label{amalgamretract}
Suppose $S$ is a homogeneous total order and $\{h_Q:S\to Q_\bot\mid Q\in\mathcal Q\}$ is a matching family of monotone maps indexed by the elements of a cover $\mathcal Q$ of $P$. Further suppose that the amalgam $h:S\to P_\bot$ for the matching family exists. If $p_s\in Q$ for some $Q\in\mathcal Q$, then $h_Q(s)=p_s$. In particular, for each  $Q\in\mathcal Q$ either $\im h\cap Q=\emptyset$ or $h_Q$ is a taxotopy retract of $h$.
\end{pro}

\begin{proof}
By definition, the element $p_s$ equals $h_Q(s)$ for some $Q\in\mathcal Q$. We claim that if $p_s\in Q'$ for some $Q'\in\mathcal Q$, then $h_{Q'}(s')=p_s$ for some $s'\leq s$.

Since $p_s\in Q\cap Q'$, the definition of the matching family says that $h_{Q\cap Q'}(s)=h_Q(s)=p_s$. Since $h_{Q'}(s^\downarrow)\cap Q\neq\emptyset$, we again have $p_s=h_{Q\cap Q'}(s)=\max\{h_{Q'}(s^\downarrow)\cap Q\}$, which in turn gives $s'\leq s$ such that $h_{Q'}(s')=p_s$. This completes the proof of the claim.

Now $p_{s'}\geq h_{Q'}(s')$ by definition but the latter equals $p_s$. Hence $p_{s'}=p_s$. Therefore $p_s=h_{Q'}(s')\leq h_{Q'}(s)\leq p_s$ implies that $h_{Q'}(s)=p_s$. The final statement is clear.
\end{proof}

It follows immediately from the proof above that each $h_Q$ can be obtained by applying the greatest lower bound construction to $h$; we denote this by $h_Q=\rho_Q(h)$. If one wants to define a total function $\rho_Q:\Pos(S,P_\bot)\to\Pos(S,Q_\bot)$, then every bounded above subset of $S$ should contain a maximum which is guaranteed if $S\subseteq\mathbb Z$.

Therefore if $S\subseteq\mathbb Z$ then the assignment $h\in\Pos(S,P_\bot)\mapsto(\rho_Q(h):S\to Q_\bot)_{Q\in\mathcal Q}$ defines a bijection between $\Pos(S,P_\bot)$ and all $\mathcal Q$-indexed matching families with amalgamation property, where the inverse takes a matching family to its amalgam. Furthermore this bijection restricts to a bijection between $\Pos(S,P)$ and the set of all matching families satisfying the following condition: for each $s\in S$ there is $Q\in\mathcal Q$ with $h_Q(s)\neq\bot$.

\begin{pro}\label{compactrestriction}
Let $S\subseteq\mathbb Z$ and $\mathcal Q$ be a chain-compact cover of $P$. Suppose $(h_Q:S\to Q_\bot)$ is a matching family with amalgamation $h:S\to P_\bot$. Then $\im h$ has no maximum element if and only if, for some $Q\in\mathcal Q$, $\im{h_Q}$ has no maximum element.
\end{pro}

\begin{proof}
Suppose $\im h$ has no maximum element. By chain-compactness of $\mathcal Q$, there exists finitely many $Q_i$ whose union covers $\im h$. There is at least one $Q_i$ such that $Q_i\cap\im h$ is cofinal in $\im h$. In other words, $\im{h_{Q_i}}$ has no maximum element.

For converse, suppose that $\max\im h$ exists. For each $Q\in\mathcal Q$, we have $h_Q(s):=\max\{h(s^\downarrow)\cap Q\}$ by Proposition \ref{amalgamretract}. Thus the sequence $(h_Q(s))_{s\in S}$ will eventually stabilize and hence $\max\im{h_Q}$ exists.
\end{proof}

Duality provides analogue of this result when $\max(-)$ is replaced by $P\cap\min(-)$.

We can think of the poset $(\mathcal Q,\subseteq)$ as a category. The assignment $Q\mapsto\Pos(S,Q_\bot)$ together with the maps $(Q'\subseteq Q)\mapsto(\rho_{Q,Q'}:\Pos(S,Q_\bot)\to\Pos(S,Q'_\bot))$ can be thought of as a diagram in $\Pos$ of shape $\mathcal Q^{op}$, i.e., a functor $F_S:\mathcal Q^{op}\to\mathrm{Pos}$.
\begin{lemma}\label{limpos}
Let $S\subseteq\mathbb Z$ and $\mathcal Q$ be a cover of a nonempty poset $P$. If every $\mathcal Q$-indexed matching family of $S$-chains has the amalgamation property, then
\begin{equation*}
\Pos(S,P_\bot)=\varprojlim F_S.
\end{equation*}
\end{lemma}

\begin{proof}
First observe that, since $S\subseteq\mathbb Z$, $(\rho_Q:\Pos(S,P_\bot)\to\Pos(S,Q_\bot))_{Q\in\mathcal Q}$ is a cone over the diagram on the right hand side.

Let $(\eta_Q:R\to\Pos(S,Q_\bot))_{Q\in\mathcal Q}$ be any cone. Then, for each $r\in R$, the family $(\eta_Q(r):S\to Q_\bot)_{Q\in\mathcal Q}$ is a matching family since whenever $Q'\subseteq Q$ we have $\eta_{Q'}(r)=\rho_{Q,Q'}(\eta_Q(r))$. Since every matching family has the amalgamation property, there is a unique amalgam $\eta(r)\in\Pos(S,P_\bot)$ for this matching family. This defines a function $\eta:R\to\Pos(S,P_\bot)$. It remains to check that this map is monotone.

Suppose $r\leq r'$ in $R$. Then $\eta_Q(r)\leq\eta_Q(r')$ for each $Q\in\mathcal Q$, i.e., $\eta_Q(r)(s)\leq\eta_Q(r')(s)$ for each $s\in S$. Hence $\eta_Q(r)(s)\leq\eta_Q(r')(s)\leq\max\{\eta_Q(r')(s):Q\in\mathcal Q\}=:\eta(r')(s)$ for each $Q$. Therefore $\eta(r)(s):=\max\{\eta_Q(r)(s):Q\in\mathcal Q\}\leq\eta(r')(s)$. This completes the proof.
\end{proof}

\begin{definition}
Let $S,P$ be any two posets. For $h,k\in\Pos(S,P_\bot)$, we say that $k\preceq^*h$ if there are $e\in\Adj S$, $f\in\Adj P$ such that $(e,f^{ext})\models k\preceq h$ where $f^{ext}\in\Adj{P_\bot}$ is the unique extension of $f$. The notation $\Lambda^*(S,P_\bot)$ will denote the posetal reflection of the preorder $(\Pos(S,P_\bot),\preceq^*)$.
\end{definition}
Since $P\rightarrowtail P_\bot\in\mathcal E$, there is an induced map $\Lambda(S,P)\to\Lambda^*(S,P_\bot)$.

Note that if $k\preceq^*h$ in $\Lambda^*(S,P_\bot)$, then $k\preceq h$ in $\Lambda(S,P_\bot)$ but the converse is not true in general.

Suppose $h,k\in\Pos(S,P\bot)$ and $k\preceq^*h$ in $\Lambda^*(S,P_\bot)$. If $\im h,\im k\subseteq Q$ for some $Q\in\Subw P$, then $k\preceq h$ in $\Lambda(S,Q)$. In particular, there is an order embedding $\Lambda(S,P)\rightarrowtail\Lambda^*(S,P_\bot)$.

If $\im h,\im k\subseteq Q_\bot$ for some $Q\in\Subw P$, then $k\preceq^*h$ in $\Lambda^*(S,Q_\bot)$.

\begin{pro}\label{retractgen}
Suppose $S\subseteq\mathbb Z$ and $Q\in\Subw P$. Given $S$-chains $\alpha,\beta$ in $P_\bot$, if $\alpha\preceq^*\beta$ in $\Lambda^*(S,P_\bot)$ then $\alpha\cap Q\neq\emptyset$ if and only if $\beta\cap Q\neq\emptyset$. Furthermore if both $\alpha\cap Q$ and $\beta\cap Q$ are nonempty, then
\begin{itemize}
\item $(\alpha\cap Q)\preceq(\beta\cap Q)$ in $\Lambda(S,Q)$;
\item $(\alpha\cap Q_\bot)\preceq^*(\beta\cap Q_\bot)$ in $\Lambda^*(S,Q)$.
\end{itemize}
\end{pro}
The proof of this proposition is similar to the proof of Proposition \ref{retractnormal} and is omitted.

\begin{cor}\label{retractmap}
Suppose $S\subseteq\mathbb Z$ and $Q\in\Subw P$. If $h,k\in\Pos(S,P_\bot)$ with $h\preceq^*k$ in $\Lambda^*(S,P_\bot)$, then $\rho_Q(h)\preceq^*\rho_Q(k)$ in $\Lambda^*(S,Q_\bot)$.
\end{cor}

\begin{proof}
The case when $\im h\cap Q,\im k\cap Q$ are nonempty are covered by the lemma above. We also know that $\im h\cap Q=\emptyset$ if and only if $\im k\cap Q=\emptyset$. In this situation, the taxotopy retracts are constant $\bot$ maps and thus are taxotopy equivalent.
\end{proof}

Now we are ready to state and prove the most general van Kampen type theorem for $\Lambda(S,-)$ whenever $S\subseteq\mathbb Z$. The theory of matching families and chain-compact covers we have developed so far enables us to recover $\Lambda(S,P)$ from the data present in the images of $\Lambda(S,-)$ at elements of a chain-compact cover of $P$.
\begin{theorem}[Seifert-van Kampen theorem for $\Lambda(S,-)$]\label{vanKampengen}
Let $S\subseteq\mathbb Z$ (and hence homogeneous) and $\mathcal Q$ be a chain-compact cover of a nonempty poset $P$. If every $\mathcal Q$-indexed matching family of $S$-chains has the amalgamation property, then for $h,h'\in\Pos(S,P_\bot)$, $h\preceq^*h'$ in $\Lambda^*(S,P_\bot)$ if and only if $\rho_Q(h)\preceq^*\rho_Q(h')$ in $\Lambda^*(S,Q_\bot)$ for each $Q\in\mathcal Q$. More formally,
\begin{equation*}
\Lambda^*(S,P_\bot)=\varprojlim\{(\Lambda^*(S,Q_\bot))_{Q\in\mathcal Q},(\Lambda^*(\rho_{Q,Q'}):\Lambda^*(S,Q_\bot)\to\Lambda^*(S,Q'_\bot))_{Q'\subseteq Q}\}.
\end{equation*}
\end{theorem}

\begin{proof}
Under the hypothesis of the theorem, we showed the existence of the limit of $F_S:\mathcal Q^{op}\to\Pos$ in Lemma \ref{limpos}. The forgetful functor $\Pos\to\Set$ is continuous since it has a left adjoint and hence the limit of $F_S$ exists in $\Set$ and it is precisely the underlying set of $\Pos(S,P_\bot)$. So we only need to show that all morphisms in the limit diagram are well-behaved with respect to the taxotopy preorder $\preceq^*$.

For any $h_Q\in\Pos(S,P_\bot)$ and $Q'\subseteq Q$ in $\mathcal Q$, we know $\rho_{Q,Q'}(h_Q)$ is a taxotopy retract of $h_Q$. Therefore if $h_Q,h'_Q\in\Pos(S,Q_\bot)$ satisfy $h_Q\preceq^*h'_Q$ in $\Lambda^*(S,Q\bot)$, then by Corollary \ref{retractmap}, we obtain that $\rho_{Q,Q'}(h_Q)\preceq\rho_{Q,Q'}(h'_Q)$. In other words, the map $\rho_{Q,Q'}$ induces a map between $\Lambda^*(S,-)$. Moreover the proof of Proposition \ref{retractgen} tells us that the adjunction on $Q'_\bot$ witnessing the latter inequality can be chosen to be the restriction of the adjunction witnessing the former.

For similar reasons the map $\rho_Q:\Pos(S,P_\bot)\to\Pos(S,Q_\bot)$ preserves taxotopy order and hence there is an induced map $\Lambda^*(S,P_\bot)\to\Lambda^*(S,Q_\bot)$. In summary, if $h\preceq^*h'$ in $\Lambda^*(S,P_\bot)$ then the matching families $(\rho_Q(h))_{Q\in\mathcal Q}$ and $(\rho_Q(h'))_{Q\in\mathcal Q}$ have component-wise taxotopy relations in respective posets where adjunctions commute with the maps $\rho_{Q,Q'}$ in the diagram.

For the converse, let $(h_Q)_{Q\in\mathcal Q}$ and $(h'_Q)_{Q\in\mathcal Q}$ be two matching families with $h_Q\preceq^*h'_Q$ in $\Lambda^*(S,Q_\bot)$ such that for each $Q'\subseteq Q$ the adjunction $f_{Q'}$ on $Q'$ witnessing $h_{Q'}\preceq^*h'_{Q'}$ is the restriction of the adjunction $f_Q\in\Adj Q$ witnessing $h_Q\preceq^*h'_Q$, i.e., the adjunctions commute with the maps $\rho_{Q,Q'}$ in the diagram. As every matching family has the amalgamation property, we obtain the amalgams $h,h'\in\Pos(S,P_\bot)$. We need to find $e\in\Adj S$ and $f\in\Adj P$ such that $(e,f^{ext})\models h\preceq^* h'$ in $\Lambda^*(S,P_\bot)$.

Since $(f_Q)_{Q\in\mathcal Q}$ is a compatible family of adjunctions indexed by the elements of a chain-compact cover, Proposition \ref{compactextension} yields a unique extension $f\in\Adj P$ such that $f\mathord{\mid}_Q=f_Q$, i.e., $f$ commutes with the maps $\rho_Q$ and $\rho_{Q,Q'}$.

Now we construct $e\in\Adj S$ with the required property. Let $I,I'$ be total orders isomorphic to the sets of distinct elements in $\im h,\im{h'}$ respectively with order isomorphisms $\psi:I\to\im h$, $\psi':I'\to\im{h'}$. Let $\chi:I\rightharpoonup I'$ denote the partial matching on $I\times I'$ that is induced by the bijection between $f$-open and $f$-closed elements in $\im h,\im{h'}$. Let $J,J'$ denote the domain and codomain of the map $\chi$. There is a natural adjunction $I\rightleftarrows I'$ with fixed sets $J$ and $J'$.

We first show that $\im h$ has no maximum element if and only if $\im{h'}$ has no maximum element. Proposition \ref{compactrestriction} states that $\im h$ has no maximum if and only if, for each $Q\in\mathcal Q$, $\im{h_Q}$ has no maximum. Since $h_Q\preceq^*h'_Q$ in $\Lambda^*(S,Q_\bot)$, by Proposition \ref{compactcomparison}, this happens if and only if, for each $Q$, $\im{h'_Q}$ has no maximum element. Again by Proposition \ref{compactrestriction}, this happens if and only if the image of the amalgam $\im{h'}$ has no maximum element.

Dually, $\im h\cap P$ has no minimum if and only if $\im{h'}\cap P$ has no minimum. Recall that $\bot\notin\im h$ if and only if for each $s\in S$ there is $Q\in\mathcal Q$ with $h_Q(s)\neq\bot$. As a consequence we obtain that $\bot\notin\im h$ if and only if $\bot\notin\im{h'}$.

Since both $I,I'$ are images of $S$ under monotone maps, there are embeddings $\phi:I\to S$ and $\phi':I'\to S$. We impose some further restrictions on $\phi$ and $\phi'$.
\begin{itemize}
\item If $\bot\in\im h\cap\im{h'}$, then $I,I'$ have a least element, say $i_\bot, i'_\bot$ respectively, then we assume that $\phi(i_\bot)=\phi'(i'_\bot)$.
\item If $\min(\im h\cap P),\min(\im{h'}\cap P)$ exist then $i_0:=\min(I\setminus\{i_\bot\})$ and $i'_0:=\min(I'\setminus\{i'_\bot\})$ exist, then we assume that $\phi(i_0)=\phi'(i'_0)$.
\item If $i_1:=\max I$ and $i'_1:=\max I'$ exist, then we assume that $\phi(i_1)=\phi'(i'_1)$.
\end{itemize}
Define $e^*:S\rightleftarrows S:e_*$ as follows.
\begin{eqnarray*}
e^*(s')&:=&\begin{cases}\chi^{-1}(\min(s'^\uparrow\cap\phi'(J')))&\mbox{if }\emptyset\neq s'^\uparrow\cap\phi'(J')\neq\phi'(J'),\\
s'&\mbox{otherwise}.\end{cases}\\
e_*(s)&:=&\begin{cases}\chi(\max(s^\downarrow\cap\phi(J)))&\mbox{if }\emptyset\neq s^\downarrow\cap\phi(J)\neq\phi(J),\\
s&\mbox{otherwise}.\end{cases}
\end{eqnarray*}
The minimum and maximum in the above definitions exist since $S\subseteq\mathbb Z$. It is clear from the construction that $e\in\Adj S$.

Now define $h_1:S\to P_\bot$ by $h_1(s):=\psi(\max(s^\downarrow\cap\phi(I)))$ and $h'_1:S\to P_\bot$ by $h'_1(s'):=\psi'(\min(s'^\uparrow\cap\phi'(I')))$. Clearly $(e,f^{ext})\models h_1\preceq h'_1$. Since $S$ is homogeneous, $\im h=\im{h_1}$ and $\im{h'}=\im{ h'_1}$ imply that $h_1\approx h$ and $h'_1\approx h'$. Since homogeneity only affects the adjunction on the domain, we can conclude $h_1\approx^*h$ and $h'_1\approx^*h'$. Hence $h\preceq h'$.
\end{proof}

Although this theorem talks about $\Lambda^*(S,P_\bot)$, we can recover $\Lambda(S,P)$ from it as there is an order embedding $\Lambda(S,P)\rightarrowtail\Lambda^*(S,P_\bot)$.

The introduction of the bottom point in each poset is necessary for finding a good representative in the taxotopy equivalence class of the taxotopy retract. One could alternatively work with a top element and obtain a similar statement. This vaguely reminds the authors of the necessity of choosing a base point to define the fundamental group (not groupoid!) in topological homotopy theory. The introduction of the taxotopy preorder $\preceq^*$ further justifies this point as, in the topological world, the base point must be preserved.

\begin{rmk}
If we consider the cover $\mathcal Q$ as a category, then we know that $P$ is the colimit of this category. The topological van Kampen theorem takes the pushout diagram of inclusions in $\Top_*$ to a pushout diagram in the category of groups. We have already seen in Theorem \ref{vanKampenlambda} that a similar statement holds for $\lambda$ but one cannot expect, entirely for cardinality reason, a colimit diagram for $\Lambda(S,-)$ when $S$ is not a singleton. We shall see an example in Corollary \ref{Parispointgen} where the set of chains in the original poset has about the same size as the product of the sizes of the sets of chains in individual components of the cover. Hence limits are more appropriate in when $S\neq\mathbf 1$.
\end{rmk}

Note that if a chain-compact cover is finite, then every matching family automatically has the amalgamation property.

The following special case of the above theorem looks more like its topological counterpart.
\begin{cor}\label{vanKampenLambdastandard}
Suppose $S\subseteq\mathbb Z$ (and hence homogeneous). If $P=Q\cup_{Q\cap Q'}Q'$ such that $\{Q,Q',Q\cap Q'\}$ is a chain-compact cover of $P$, then there is a pullback diagram of the $\rho$s
\begin{equation*}
\Lambda^*(S,P_\bot)=\Lambda^*(S,Q_\bot)\times_{\Lambda^*(S,(Q\cap Q')_\bot)}\Lambda^*(S,Q'_\bot).
\end{equation*}
In particular,
\begin{eqnarray*}
\lambda^*(P_\bot)&=&\lambda^*(Q_\bot)\times_{\lambda^*((Q\cap Q')_\bot)}\lambda^*(Q'_\bot),\\
L^*(P_\bot)&=&L^*(Q_\bot)\times_{L^*((Q\cap Q')_\bot)}L^*(Q'_\bot).
\end{eqnarray*}
\end{cor}

The finiteness of the chain-compact cover is crucial in the above corollary.
\begin{cor}
Let $P$ be a finite poset and $\mathcal Q$ be any chain-compact cover of $P$. Then $\Lambda^*(S,P_\bot)=\varprojlim(\Lambda^*(S,Q_\bot))_{Q\in\mathcal Q}$.
\end{cor}

In the following result we use a construction motivated by the sum of ordinals and work with a cover that is not chain-compact.
\begin{por}
For $P_1,P_2\in\Pos$, let $P_1\lhd P_2$ be the poset that admits a monotone bijection $P_1\sqcup P_2\to P_1\lhd P_2$ and where each element of $P_1$ is a strictly below each element of $P_2$. Then $\Adj{P_1\lhd P_2}=\Adj{P_1}\times\Adj{P_2}$ and hence
\begin{eqnarray*}
\Lambda(\mathbb N^{op}\lhd\mathbb N,P_1\lhd P_2)&=&(\Lambda(\mathbb N^{op},P_1)\times\Lambda(\mathbb N,P_2))\sqcup\Lambda(\mathbb N^{op}\lhd\mathbb N,P_1)\\&&\sqcup\Lambda(\mathbb N^{op}\lhd\mathbb N,P_2).
\end{eqnarray*}
\end{por}

\section{Finding covers and rigid subsets}
So far we have proved two results for the computation of fundamental posets of a poset $P$ both of which rely on the existence of proper subsets in $\Subw P$. In this section we construct two decreasing chains of elements of $\Subw P$ and compute fundamental posets in some concrete examples. Another goal of this section is to find two ``rigid'' subsets of the poset which always embed into the fixed set of the closure/interior operator induced by an adjunction on a poset. We always assume that our posets are finite but the results could be generalized to infinite posets subject to some finiteness conditions.

In a bounded finite poset $P$, the special elements $\top$ and $\bot$ can be thought of as $\bigwedge\emptyset$ and $\bigvee\emptyset$ respectively, and hence such elements are preserved by right and left adjoints respectively. The following definitions describe the iterated generalized empty meets and joins. The preservation properties for such elements under adjoints will be shown in Porism \ref{rigid}.
\begin{definitions}
For a finite connected poset $P$, the notations $\max P$ and $\min P$ denote the subsets of $P$ containing all maximal and minimal elements respectively. Define
\begin{equation*}
T(P):=\bigcap_{x\in\max P}x^\downarrow,\quad B(P):=\bigcap_{y\in\min P}y^\uparrow.
\end{equation*}
Set $T_0(P)=B_0(P):=P$ and, for each $n\geq 1$, define $T_{n+1}(P):=T(T_n(P))$ and $B_{n+1}(P):=B(B_n(P))$.

Since $P$ is finite, the sequences $(T_n(P))_{n\geq 0}$ and $(B_n(P))_{n\geq 0}$ must stabilize.

Define $t(P):=\min\{n\mid T_n(P)=T_{n+1}(P)\mbox{ or }T_{n+1}(P)=\emptyset\}$ and $b(P):=\min\{n\mid B_n(P)=B_{n+1}(P)\mbox{ or }B_{n+1}(P)=\emptyset\}$.
\end{definitions}
We drop reference to the poset $P$ if it is clear from the context.

The following result describes the images of maximal elements under left adjoints.
\begin{pro}
Let $P$ be a finite connected poset and let $f\in\Adj P$. Then there is a permutation $\eta:\max P\to\max P$ such that, for all $a,b\in\max P$, $f^*(a)\leq b$ if and only if $b=\eta(a)$. Dually, there is a permutation $\rho:\min P\to\min P$ such that, for all $a,b\in\min P$, $f_*(b)\geq a$ if and only if $a=\rho(b)$.
\end{pro}

\begin{proof}
We just show the statement for the maximal elements; the proof for the minimal elements follows by duality.

Let $a_1,a_2\in\max P$. If $(f^*a_1)^\uparrow\cap(f^*a_2)^\uparrow\neq\emptyset$, then there is some $b\in\max P$ such that $f^*a_1,f^*a_2\leq b$. Applying $f_*$ we obtain $a_1,a_2\leq f_*b$ since $f_*f^*a_i=a_i$ for $i=1,2$. This is impossible unless $a_1=a_2$ since both are maximal. Since $\max P$ is finite, an application of the pigeonhole principle completes the proof.
\end{proof}

Now we fulfill our promise to provide a way to compute the fundamental poset of a finite connected poset using the knowledge of the fundamental posets of its subposets.
\begin{theorem}\label{Parispoint}
Let $P$ be a finite connected poset with $t:=t(P)$ and $b:=b(P)$. Suppose $T_t(P)=\{x\}$ and $\min B_b(P)=\{y\}$ are cutsets with $x\geq y$. Then
\begin{equation*}
\lambda(P)=\lambda(T_t)\cup_{\lambda(T_t\cap B_b)}\lambda(B_b)=\lambda(x^\downarrow)\cup_{\lambda([y,x])}\lambda(y^\uparrow).
\end{equation*}
\end{theorem}

\begin{proof}
Let $\{z\}$ be a cutset in a poset $Q$. Then the inclusions $z^\downarrow\rightarrowtail Q$ and $z^\uparrow\rightarrowtail Q$ have the extension property for one can extend an adjunction on the domain by identity maps.

Note that $\{x\}$ is a cutset in $P$ as well as in $B_b$; similarly $\{y\}$ is a cutset in $P$ as well as in $T_t$. Moreover $T_t\cap B_b=[y,x]$. Hence all four inclusions $[y,x]\rightarrowtail x^\downarrow\rightarrowtail P$, $[y,x]\rightarrowtail y^\uparrow \rightarrowtail P$ have extension.

\textbf{Claim:} $T(P)\in\Subw P$.

Let $f\in\Adj P$ and $p\in T(P)$. We want to show that $f^*p,f_*p\in T(P)$.

Since $p\in T(P)$, we have $p\leq a$ for all $a\in\max P$. Hence $f^*p\leq f^*a$ for all $a\in\max P$. From the above proposition we obtain a unique upper bound for $f^*a$ in $\max P$ and for each $b\in\max P$ there is a unique $a\in\max P$ such that $f^*a\leq b$. Hence we conclude that $f^*p\leq b$ for all $b\in\max P$. This shows that $f^*p\in T(P)$.

Now we show that $f_*p\in T(P)$. Since $p\leq b$ for all $b\in\max P$, we get $f_*p\leq f_*b$ for all $b\in\max P$. But $\{f_*b:b\in\max P\}=\max P$. Hence $f_*p\leq a$ for all $a\in\max P$. This completes the proof of the claim.

Since $T_{n+1}(P)=T(T_n(P))$, we get $T_{n+1}(P)\in\Subw{T_n(P)}$ for each $0\leq n\leq t-1$ using a similar argument. Since $\mathcal R$ is a category, the inclusion $T_t\rightarrowtail P$ has the restriction property. Dually we also obtain that $B_b\in\Subw P$.

Finally note that $x\geq y$ implies the equalities $T_n(B_b)=T_n(P)\cap B_b$ and $B_n(T_t)=B_n(P)\cap T_t$ for each $n\geq 0$. In particular, $T_t(B_b)=[y,x]=T_t\cap B_b=B_b(T_t)$. Hence $[y,x]\in\Subw{T_t}\cap\Subw{B_b}$.

Applying Theorem \ref{vanKampenlambda} to the pushout $P=T_t\cup_{T_t\cap B_b}B_b$ where all inclusions have extension as well as restriction properties, we get the required pushout diagram of their $\lambda$s.
\end{proof}

The following generalization of the above has essentially the same proof.
\begin{theorem}
Let $P$ be a finite connected poset with $t:=t(P)$ and $b:=b(P)$. Suppose the following conditions are satisfied for some $n\leq t$ and $m\leq b$.
\begin{itemize}
\item The antichains $\max T_n$ and $\min B_m$ are cutsets in $P$.
\item For each $p\in\min B_m$ there is $q\in\max T_n$ such that $p\leq q$.
\item The inclusions $T_n\rightarrowtail P$ and $B_m\rightarrowtail P$ have the extension property.
\end{itemize}
Then $\lambda(P)=\lambda(T_n)\cup_{\lambda(T_n\cap B_m)}\lambda(B_m)$.
\end{theorem}

The following result guarantees the existence of a certain poset in each fixed set.
\begin{por}[Rigid subsets]\label{rigid}
Let $P$ be a finite connected poset. Given any $f\in\Adj P$ the closure operator $f_*f^*$ fixes $\bigsqcup_{0\leq n\leq t}\max T_n(P)$ pointwise and the interior operator $f^*f_*$ fixes $\bigsqcup_{0\leq m\leq b}\min B_m(P)$ pointwise.
\end{por}

\begin{proof}
In the proof of Theorem \ref{Parispoint} we showed that for any finite connected poset $Q$, $T(Q)\in\Subw Q$. Iterating the same argument an adjunction on $P$ restricts to an adjunction on $T_n(P)$ for any $0\leq n\leq t$.

Recall that for any adjunction $g$ on a finite connected poset $Q$, the closure operator $g_*g^*$ fixes $\max Q$ pointwise. Combining the above two statements we observe that $f_*f^*$ fixes $\max T_n(P)$ pointwise. The second statement follows by duality.
\end{proof}

The following result follows immediately from Corollary \ref{vanKampenLambdastandard} and the proof of Theorem \ref{Parispoint}.
\begin{cor}\label{Parispointgen}
Let $S\subseteq\mathbb Z$ and $P$ be a finite connected poset with $t:=t(P)$ and $b:=b(P)$. Suppose the following conditions are satisfied for some $n\leq t$ and $m\leq b$.
\begin{itemize}
\item The antichains $\max T_n$ and $\min B_m$ are cutsets in $P$.
\item For each $p\in\min B_m$ there is $q\in\max T_n$ such that $p\leq q$.
\end{itemize}
Then $\Lambda^*(S,P_\bot)=\Lambda^*(S,(T_n)_\bot)\times_{\Lambda^*(S,(T_n\cap B_m)_\bot)}\Lambda^*(S,(B_m)_\bot)$.
\end{cor}

\section{Future directions}
An algebra is usually defined to be a structure with functions that lacks relations. Taxotopy theory could be loosely thought of as a category-theoretic version of homotopy theory minus algebra. As of now we are not aware of any applications of taxotopy theory, mainly because it is not well-developed yet! We are continuing to develop it, with further results to appear in a forthcoming paper.

The order complex $\Delta(P)$ of a finite poset $P$ is a finite simplicial complex. Given a finite simplicial complex $K$, one can construct the poset $\mathcal P(K)$ of inclusions of faces. The complexes $K$ and $\Delta(\mathcal P(K))$ are homotopy equivalent but unfortunately the partial order structures on $P$ and $\mathcal P(\Delta(P))$ are unrelated. Bergman \cite{Berg} recently provided a way to lift the partial order on $P$ to a partial order structure on the simplicial complex $\Delta(P)$. Building on this idea, we will provide some results on simplicial taxotopy theory in the next paper.

Another topic we will address in a future work is the study of different topologies on the set of adjunctions on a poset/category. There are at least three natural topologies on $\Adj P$ which interact well with the adjoint topology on $P$.

Owing to the similarities of taxotopy theory with homotopy theory, it is natural to ask whether a given construction, property or a theorem in homotopy theory has an analogue in taxotopy theory. Some of the topics for which we would like to explore possible taxotopy-theoretic analogues include covering spaces, higher fundamental groups, long exact sequences of fundamental groups, and---most importantly---model structures.

Sometimes we cannot expect any analogue for trivial reasons. For example the fundamental group preserves finite products. We cannot expect analogous statement for $\lambda(-)$ because $\Adj{P\times Q}$ is much bigger than $\Adj P\times\Adj Q$.

As a byproduct of the theory we developed, we found some results on fixed points and fixed sets in the section on rigid subsets; fixed point theory has been a topic of interest in topology as well as in order theory for over a century. It will be worth investigating the connections between fixed-point theory and taxotopy theory further.

Lastly, we ask whether ideas from taxotopy theory can be applied to residuated monoids. A residuated monoid is an ordered monoid $M$ with the property that, for each $a\in M$, the (monotone) multiplication functions $a\cdot(-)$ and $(-)\cdot a$ from $M$ to $M$ have named right adjoints; see \cite{BlythJanowitz}. Preliminary investigations suggest that if the definition of the taxotopy preorder on $M$ was strengthened to require the existence of these ``algebraic'' adjunctions, the resulting preorder on $M$ would not necessarily be trivial, and could provide interesting information about the structure of $M$.

\vspace{8pt}
\begin{flushleft}
Amit Kuber\\
Dipartimento di Matematica e Fisica\\
Seconda Universit\'{a} degli Studi di Napoli\\
Email: \texttt{expinfinity1@gmail.com}
\end{flushleft}

\vspace{8pt}
\begin{flushleft}
David Wilding\\
Email: \texttt{david@dpw.me}\\
\end{flushleft}
\end{document}